\documentclass[reqno]{amsart}
\usepackage[utf8]{inputenc}

\usepackage{amsfonts, amsmath,  amssymb, color,colortbl, enumerate, esint, fancyvrb, mathrsfs, tikz}

\usepackage{pgfplots, subcaption}

\pgfplotsset{compat=1.10}
\usepgfplotslibrary{fillbetween}
\usetikzlibrary{patterns}

\pgfdeclarelayer{ft}
\pgfdeclarelayer{bg}
\pgfsetlayers{bg,main,ft}

\VerbatimFootnotes

\usetikzlibrary{positioning}
\VerbatimFootnotes

\newcommand{\R}{\mathbb{R}}

\newcommand{\N}{\mathbb{N}}

\newcommand{\T}{\mathbb{T}}

\renewcommand{\O}{\mathcal{O}}
\newcommand{\Sc}{\mathcal{S}}

\newcommand{\fh}{\mathfrak{h}}

\newcommand{\bZ}{\mathbf{Z}}
\newcommand{\bY}{\mathbf{Y}}

\newcommand{\cB}{\mathcal{B}}

\newcommand{\sE}{\mathscr{E}}

\newcommand{\ta}{\texttt{a}}
\newcommand{\bta}{\mbox{\small$\ta$}}
\newcommand{\sta}{\mbox{\scriptsize$\ta$}}
\newcommand{\tc}{\texttt{c}}
\newcommand{\btc}{\mbox{\small$\tc$}}
\newcommand{\stc}{\mbox{\scriptsize$\tc$}}

\newcommand{\BL}[1]{\mathrm{BL}_{#1}}
\newcommand{\tang}{\mathrm{tang}}
\newcommand{\trans}{\mathrm{trans}}
\newcommand{\cell}{\mathrm{cell}}
\newcommand{\alg}{\mathrm{alg}}
\newcommand{\Deg}{\overline{\deg}\,}

\renewcommand{\ge}{\geqslant}
\renewcommand{\leq}{\leqslant}
\renewcommand{\geq}{\geqslant}

\theoremstyle{plain}
\newtheorem{theorem}{Theorem}[section]
\newtheorem{lemma}[theorem]{Lemma}
\newtheorem{proposition}[theorem]{Proposition}
\newtheorem{corollary}[theorem]{Corollary}
\newtheorem{conjecture}[theorem]{Conjecture}
\newtheorem*{first key estimate}{First key estimate}
\newtheorem*{second key estimate}{Second key estimate}

\theoremstyle{definition}
\newtheorem{definition}[theorem]{Definition}

\newtheorem*{acknowledgement}{Acknowledgement}
\newtheorem*{induction hypothesis}{Induction hypothesis}
\newtheorem*{notation}{Notation}

\title[New Kakeya estimates]{New Kakeya estimates using the\\ polynomial Wolff axioms}

\author{Jonathan Hickman}
\address{Mathematical Institute, University of St Andrews, North Haugh,
St Andrews, Fife, KY16 9SS, UK.}
\email{jeh25@st-andrews.ac.uk}
\author{Keith M. Rogers}
\address{Instituto de Ciencias Matem\'aticas CSIC-UAM-UC3M-UCM, Madrid, 28049, Spain.}
\email{keith.rogers@icmat.es}

\thanks{Partially supported by the EPSRC grant EP/R015104/1, the NSF grant DMS-1440140, and the MINECO grants SEV-2015-0554 and MTM2017-85934-C3-1-P}

\begin{document}
\begin{abstract} We obtain new bounds for the Kakeya maximal conjecture in most dimensions $n<100$, as well as improved bounds for the Kakeya set conjecture when $n=7$ or $9$. For this we consider
Guth and Zahl's strengthened formulation of the maximal conjecture, concerning families of tubes that satisfy the polynomial Wolff axioms. Our results give improved estimates for this strengthened formulation when $n =5$ or $n \geq 7$.  \end{abstract}

\maketitle




\section{Introduction}

For $n \geq 2$ and small $\delta>0$, a \emph{$\delta$-tube} is a cylinder $T\subset \mathbb{R}^n$ of unit height and radius~$\delta$, with arbitrary position and arbitrary orientation $\mathrm{dir}(T)\in S^{n-1}$. A family~$\T$ of $\delta$-tubes is \emph{direction-separated} if $\{\mathrm{dir}(T) : T \in \T\}$ forms a $\delta$-separated subset of the unit sphere.

\begin{conjecture}[Kakeya maximal conjecture]\label{maximal conjecture} Let  $p\ge \frac{n}{n-1}$. For all $\varepsilon > 0$, there exists a constant $C_{\varepsilon, n}>0$ such that
\begin{equation}\label{maximal inequality}
\Big\| \sum_{T\in\mathbb{T}} \chi_T\Big\|_{L^p(\mathbb{R}^n)}\,\leq\,  C_{\varepsilon, n} \delta^{-(n-1 - n/p)-\varepsilon} \Big(\sum_{T \in \T} |T| \Big)^{1/p}
\end{equation}
whenever $0 < \delta < 1$ and $\mathbb{T}$ is a direction-separated family of $\delta$-tubes. 
\end{conjecture} 

By an application of H\"older's inequality, one may readily verify that if \eqref{maximal inequality} holds for $p = \frac{n}{n-1}$, then, for all $\varepsilon>0$,  there exists a constant $c_{\varepsilon, n}>0$ such that
\begin{equation*}
    \big|\bigcup_{T \in \T} T \big| \,\geq\, c_{\varepsilon,n} \delta^{\varepsilon} \sum_{T \in \T} |T|.
\end{equation*}
This can be interpreted as the statement that any direction-separated family of  $\delta$-tubes is `essentially disjoint'. A more refined argument shows that if \eqref{maximal inequality} holds for a given $p$, then every Kakeya set in $\R^n$ (that is, every compact set that contains a unit line segment in every direction) has Hausdorff dimension at least~$p'$, the conjugate exponent of $p$. Thus, Conjecture~\ref{maximal conjecture} would imply the \emph{Kakeya set conjecture}, that Kakeya sets in~$\R^n$  have Hausdorff dimension~$n$; see, for instance, \cite{Bourgain1991a, Wolff1999, KT2002b}.

For $n=2$, the set conjecture was proven by Davies \cite{Davies1971} and the maximal conjecture was proven by C\'ordoba~\cite{Cordoba1977} in the seventies. Both conjectures remain challenging and important open problems in higher dimensions; for partial results, see~\cite{CDR1986, Bourgain1991a, Wolff1995, Schlag1998, TVV1998, Bourgain1999, KT1999, KLT2000, KT2002,  BCT2006, Dvir2009, Guth2010, Guth2016b, GR, KZ2019} and references therein. We highlight, in particular, the classical work of Wolff \cite{Wolff1995}, which considers more general families of tubes satisfying the following hypothesis.

\begin{definition} We say that $\T$ satisfies the \emph{linear Wolff axiom} if there is a constant $N \geq 1$, depending only on $n$, such that
\begin{equation*}
    \#\big\{ T \in \mathbb{T} :  T \subseteq E \big\} \,\leq\,  N\delta^{-(n-1)} |E|
\end{equation*}
whenever $E \subset \R^n$ is a rectangular box of arbitrary dimensions.
\end{definition}

In \cite{Wolff1995}, Wolff showed that \eqref{maximal inequality} holds for the restricted range $p \geq \frac{n+2}{n}$ whenever~$\T$ satisfies the linear Wolff axiom.\footnote{Strictly speaking, Wolff's theorem \cite{Wolff1995} holds under a slightly less restrictive condition referred to simply as \emph{the Wolff axiom}. See \cite{GZ2018} for a comparison of these conditions.} Furthermore, it is not difficult to see that any direction-separated $\T$ satisfies the linear Wolff axiom and so his result provides similar progress for  Conjecture \ref{maximal conjecture}. 

Interestingly, there exist examples of tube families $\T$ in dimensions $n \geq 4$ that satisfy the linear Wolff axiom, but for which \eqref{maximal inequality} fails to hold for the whole range $p \geq \frac{n}{n-1}$; see \cite{Tao2005}. In particular, when $n = 4$ one may construct such $\T$ for which~\eqref{maximal inequality} is only valid in Wolff's range $p \geq 3/2$. Examples of this kind are not direction-separated and therefore do not provide counterexamples to Conjecture~\ref{maximal conjecture}. 

To go beyond $p \geq 3/2$ in four dimensions, Guth and Zahl \cite{GZ2018} considered families of tubes which satisfy a more restrictive version of the linear Wolff axiom.

\begin{definition}
We say that $\mathbb{T}$ satisfies the {\it $(D,N)$-polynomial Wolff axiom} if  
\begin{equation*}
    \#\big\{ T \in \mathbb{T} :  |T \cap E| \geq \lambda |T| \big\} \,\leq\,  N\delta^{-(n-1)} \lambda^{-n} |E|
\end{equation*}
whenever $\lambda\ge \delta$ and $E\subset \R^n$ is a semialgebraic set of complexity at most $D$.
\end{definition}

In \cite{KR2018}, Katz and the second author showed that for all $D\in\mathbb{N}$ and all $\varepsilon>0$ there is a constant $C_{\varepsilon,n,D}$ such that any direction-separated family $\mathbb{T}$ satisfies the $(D,N)$-polynomial Wolff axiom with $N = C_{\varepsilon,n,D}\delta^{-\varepsilon}$ (see also~\cite{Guth2016} and~\cite{Zahl2018} for similar results in three and four dimensions, respectively). Thus, the  following conjecture of Guth and Zahl~\cite[Conjecture 1.1]{GZ2018} is stronger than the Kakeya maximal conjecture.

\begin{conjecture}[Guth--Zahl \cite{GZ2018}]\label{PWAcon} Let   $p\ge \frac{n}{n-1}$. For all $\varepsilon > 0$, there is a complexity $D = D_{\varepsilon,n} \in \mathbb{N}$ and a constant $C_{\varepsilon,n} >0$ such that
\begin{equation}\label{kakeya}\tag{$\mathrm{K}_p$}
\Big\| \sum_{T\in\mathbb{T}} \chi_T\Big\|_{L^p(\mathbb{R}^n)}\,\leq\,  C_{\varepsilon,n} N^{1-1/p}\delta^{-(n-1 - n/p)-\varepsilon}\Big(\sum_{T \in \T} |T| \Big)^{1/p}
\end{equation}
whenever $0 < \delta < 1$, $N \geq 1$ and $\mathbb{T}$ satisfies the $(D,N)$-polynomial Wolff axiom.\footnote{Strictly speaking, the conjecture of \cite{GZ2018} is slightly weaker than Conjecture \ref{PWAcon} in some regards and stronger in others. The positive results of \cite{GZ2018} are also stated in a slightly different form.}
\end{conjecture} 

It is easy to adapt C\'ordoba's $L^2$ argument \cite{Cordoba1977} to prove Conjecture \ref{PWAcon} for $n=2$. Guth and Zahl \cite{GZ2018} showed that in four dimensions, under the polynomial Wolff axiom, the $p \geq 3/2$ bound can be improved to $p\ge 85/57$.$^2$ In all other dimensions the Wolff bound $p \geq \frac{n+2}{n}$ provides the previous best known result under the polynomial Wolff axioms alone.  Our main result improves this range in high dimensions.

\begin{theorem}\label{linear theorem} Conjectures \ref{maximal conjecture} and \ref{PWAcon} are true in the range $p \ge p_n$, where
\begin{equation}\label{linear exponent}
p_n := 1+\min_{2\leq k\leq n}\max\Big\{\, \Big(\frac{n}{n-1}\Big)^{n-k}, \,\frac{n-1}{n-k+1}\,\Big\}\frac{1}{n-1}.
\end{equation}
\end{theorem}

The range of exponents $p \geq p_n$ is larger than Wolff's when $n=5$ or $n \geq 7$. To see this, note that for any $0 < r < 1$ there exists some integer $2 \leq k \leq n$ satisfying $k\in[r(n-1)+1,r(n-1)+2)$. Writing $p_n = 1 + \alpha_n\frac{1}{n-1}$, it follows that
\begin{equation*}
\alpha_n < \inf_{0< r < 1}\max\Big\{\,\Big(1+\frac{1}{n-1}\Big)^{(n-1)(1-r)},\, \frac{1}{1-r}\,\Big\} \,\leq\,  \Omega^{-1}=1.763....  
\end{equation*}
Here the omega constant $\Omega\in(1/2,1)$ is the solution to $e^{\Omega}=\Omega^{-1}$. In particular, Theorem~\ref{linear theorem} implies that Conjecture~\ref{PWAcon} is true in the range $p \geq 1+\Omega^{-1}\frac{1}{n-1}$, yielding an improvement over Wolff's bound when $n \geq 9$. Calculating the precise value of $p_n$ for lower $n$, we find that Theorem~\ref{linear theorem} also improves the state-of-the-art for Conjecture~\ref{PWAcon} in dimensions $n=5, 7, 8$; see Figure~\ref{exponent table} for explicit values of~$p_n$.

 On the other hand, Katz and Tao~\cite{KT2002} confirmed Conjecture~\ref{maximal conjecture}  in the range $p\ge 1+\frac{7}{4}\frac{1}{n-1}$. One may refine the above observations to show that $\alpha_n \to \Omega^{-1}$ as $n$ tends to infinity and so, as $7/4=1.7 5< \Omega^{-1}$,  Theorem~\ref{linear theorem} does not constitute  an improvement for Conjecture~\ref{maximal conjecture} in high dimensions. However, by explicitly
calculating values of $p_n$, we obtain improvements for Conjecture~\ref{maximal conjecture}  in all dimensions~$n$ belonging to the following list:
 \begin{quote} 5, 7, 8, 9, 10, 11, 12, 13, 14, 15, 16, 17, 18, 19, 20, 21, 22, 23, 24, 25, 26, 27, 28, 30, 31, 32, 33, 34, 35, 37, 39, 40, 41, 42, 44, 46, 47, 48, 49, 51, 53, 55, 56, 58, 60, 62, 65, 67, 69, 72, 74, 76, 81, 83, 90,~97.\footnote{This list was compiled using the following \emph{Maple} \cite{Maple} code:

\begin{verbatim}
printlevel := 0: N := [insert dimension]:
p_broad := 1+(n/(n-1))^(n-k)/(n-1): p_limit := 1+1/(n-k+1):
p_Wolff := (n+2)/n: p_KT := 1+(7/4)/(n-1):

for n from 2 to N do
p_seq := [seq(max(eval(p_broad, k = i), eval(p_limit, k = i)), i = 2 .. n)]:
new_exponent := min(p_seq):
if new_exponent < min(p_Wolff, p_KT) then print(n) end if:
end do:
\end{verbatim}\vspace{-1em}}
\end{quote}

Finally, recall that maximal estimates imply bounds for the dimension of Kakeya sets, and in particular we obtain the following corollary.

\begin{corollary}\label{Kakeya set bounds} Kakeya sets in $\mathbb{R}^7$ have Hausdorff dimension at least $1+6^4/7^3$ and Kakeya sets in $\mathbb{R}^9$ have Hausdorff dimension at least $1 + 8^5/9^4$.
\end{corollary}

Corollary \ref{Kakeya set bounds} improves the previous best known lower bound of $(2-\sqrt{2})(n-4)+3$, also due to Katz and Tao~\cite{KT2002} (note that this bound is better than the one that can be obtained from their maximal estimate).

As noted above, in high dimensions the Kakeya maximal bounds of Katz--Tao~\cite{KT2002} are asymptotically superior to those of Theorem \ref{linear theorem}. Nevertheless, the bounds for Conjecture \ref{maximal conjecture} given by Theorem \ref{linear theorem} are perhaps of interest in this regime since they are obtained via a completely different approach from that used in \cite{KT2002}. The Katz--Tao~\cite{KT2002} bound is proved using the \emph{sum-difference method} from additive combinatorics, building upon earlier work of Bourgain \cite{Bourgain1999}.\footnote{The sum-difference approach heavily exploits the direction separation hypothesis and therefore does not appear to yield estimates under the more general polynomial Wolff axiom hypothesis.} The proof of Theorem~\ref{linear theorem}, by contrast, is based on the~\emph{polynomial partitioning method}. This method was introduced by Guth and Katz \cite{GK2015} in their resolution of the Erd\H{o}s distance conjecture, and was inspired by Dvir's solution to a finite field analogue of the Kakeya problem~\cite{Dvir2009}. In recent years, polynomial partitioning techniques have been substantially developed so as to apply to a wide variety of problems in combinatorics and harmonic analysis; see, for instance, \cite{TS2012, Guth2016, Guth2019, Wang2018, HR2018}. In the present article,  an argument of Guth \cite{Guth2016, Guth2019}, which was previously used to study oscillatory integral operators, is adapted so as to directly apply to the Kakeya problem. The structure of the proof and the presentation of the paper follows closely that of the companion article \cite{HR2018}, where Guth's arguments were extended in the oscillatory integral context so as to take into account polynomial Wolff axiom information in all dimensions.

\renewcommand{\arraystretch}{1.3}
\renewcommand{\arraystretch}{1.3}
 \begin{figure}
\begin{tabular}{ ||c|c|c||c|c|c|| } 
 \hline
  $n$ & $p_n$  & $p\ge p_n$ & $n$ & $p_n$  & $p\ge p_n$ \\ 
   \hline 
 \cellcolor{white} 2 & 2 & C\'ordoba~\cite{Cordoba1977}  &  \cellcolor{yellow!50} 9 &  \cellcolor{yellow!50} $1 + 9^4/8^5$ &  \cellcolor{yellow!50} Theorem~\ref{linear theorem}\\ 
  \cellcolor{white} 3 & $5/3$ & Wolff~\cite{Wolff1995}  &\cellcolor{yellow!50}  10 & \cellcolor{yellow!50}   $1+10^5/9^6$ &  \cellcolor{yellow!50} Theorem~\ref{linear theorem}  \\
\cellcolor{white} 4 &   $85/57$ &   Guth--Zahl~\cite{GZ2018}  & \cellcolor{yellow!50} 11 &  \cellcolor{yellow!50} $7/6$ &  \cellcolor{yellow!50} Theorem~\ref{linear theorem}\\  
\cellcolor{yellow!50} 5 & \cellcolor{yellow!50} $1+5^2/4^3$ & \cellcolor{yellow!50} Theorem~\ref{linear theorem}    & \cellcolor{yellow!50}  12 & \cellcolor{yellow!50}   $1+12^6/11^7$ &  \cellcolor{yellow!50} Theorem~\ref{linear theorem}   \\ 
\cellcolor{white} 6 & 4/3 & Wolff~\cite{Wolff1995}  &   \cellcolor{yellow!50}  13 &  \cellcolor{yellow!50} $8/7$ &  \cellcolor{yellow!50} Theorem~\ref{linear theorem}\\ 
 \cellcolor{yellow!50} 7 &  \cellcolor{yellow!50} $1+7^3/6^4$ &  \cellcolor{yellow!50} Theorem~\ref{linear theorem}  &  \cellcolor{yellow!50}  14 & \cellcolor{yellow!50}   $1+14^7/13^8$ &  \cellcolor{yellow!50} Theorem~\ref{linear theorem}   \\ 
\cellcolor{yellow!50}  8 & \cellcolor{yellow!50}   $1+8^4/7^5$ &  \cellcolor{yellow!50} Theorem~\ref{linear theorem}   &  \cellcolor{yellow!50}  15 & \cellcolor{yellow!50}   $1+15^8/14^9$&  \cellcolor{yellow!50} Theorem~\ref{linear theorem}\\ 
  \hline
\end{tabular}
    \caption{The current state-of-the-art for Conjectures~\ref{maximal conjecture}  and~\ref{PWAcon} in low dimensions. New results are \colorbox{yellow!50}{highlighted} and are deduced by combining Theorem~\ref{main theorem} with Proposition~\ref{Bourgain--Guth}.
    }
    \label{exponent table}
\end{figure}

The article is organised as follows:
\begin{itemize}
    \item After fixing some notation in Section \ref{notation section}, in Section \ref{reduction section} the problem is reduced to estimating the so-called \emph{$k$-broad norms} for the Kakeya maximal function, paralleling work on oscillatory integrals from~\cite{BG2011, Guth2016, Guth2019, HR2018}.
    \item In Sections~\ref{k-broad norms section} and \ref{polynomial partitioning section}, the basic tools for the proof of Theorem~\ref{linear theorem} are recalled from the literature and, in particular, the theory of $k$-broad norms and the polynomial partitioning theorem from~\cite{Guth2019} are reviewed.
   \item In Section~\ref{structure lemma section}, a recursive algorithm is described which can be interpreted as a structural statement of algebraic nature concerning extremal configurations of tubes for the Kakeya problem. 
    \item In Section~\ref{proof section}, the polynomial Wolff axioms are combined with the recursive algorithm to conclude the proof.
\end{itemize}

\begin{acknowledgement} The first author thanks both Larry Guth and Joe Karmazyn for helpful discussions during the development of this project.  
\end{acknowledgement}




\section{Notational conventions}\label{notation section}

 We call an $n$-dimensional ball $B_r$ of radius~$r$ an {\it $r$-ball}. The intersection of $S^{n-1}$ with a ball is called a \textit{cap}. The $\delta$-neighbourhood of a set $E$ will be denoted by $N_{\;\!\!\delta}E$.

The arguments will involve the {\it admissible parameters} $n$, $p$ and $\varepsilon$ and the constants in the estimates will be allowed to depend on these quantities. Moreover, any constant is said to be \textit{admissible} if it depends only on the admissible parameters. Given positive numbers $A, B \geq 0$ and a list of objects $L$, the notation $A \lesssim_L B$, $B \gtrsim_L A$ or $A = O_L(B)$ signifies that $A \leq C_L B$ where $C_L$ is a constant which depends only on the objects in the list and the admissible parameters. We write $A \sim_L B$ when both $A \lesssim_L B$ and $B \lesssim_L A$.

 The cardinality of a finite set $A$ is denoted by $\#A$. A set $A'$ is said to be a \emph{refinement} of $A$ if $A' \subseteq A$ and $\#A' \gtrsim \#A$. In many cases it will be convenient to \emph{pass to a refinement} of a set $A$, by which we mean that the original set $A$ is replaced with some refinement. 




\section{Reduction to $k$-broad estimates}\label{reduction section}

Rather than attempt to prove~\eqref{kakeya} directly, it is useful to work with a class of weaker inequalities known as \emph{$k$-broad estimates}. This type of inequality was introduced by Guth~\cite{Guth2016, Guth2019} in the context of oscillatory integral operators (and, in particular, the Fourier restriction conjecture) and was inspired by the earlier multilinear theory developed in~\cite{BCT2006} (see also \cite{Bennett2014} for a detailed discussion of multilinear Kakeya inequalities or Proposition \ref{k-broad vrs k-linear} below for a precise statement relating the $k$-broad and $k$-linear theory). 

In order to introduce the $k$-broad estimates, we decompose the unit sphere $S^{n-1}$ into finitely-overlapping caps $\tau$ of diameter~$\beta$, an admissible constant satisfying $\delta \ll \beta \ll 1$. We then perform a corresponding decomposition of $\T$ by writing the family as a disjoint union of subcollections
\begin{equation*}
    \T = \bigcup_{\tau} \T[\tau]
\end{equation*} 
where each $\T[\tau]$ satisfies $\mathrm{dir}(T) \in \tau$ for all $T \in \T[\tau]$. The ambient euclidean space is also decomposed into tiny balls $B_{\;\!\!\delta}$ of radius~$\delta$. In particular, fix $\mathcal{B}_{\delta}$ a collection of finitely-overlapping $\delta$-balls which cover $\R^n$. For $B_{\;\!\!\delta} \in \mathcal{B}_{\delta}$ define
\begin{equation*}
\mu_{\T}(B_{\;\!\!\delta}) := \min_{V_1,\dots,V_{\!A} \in \mathrm{Gr}(k-1, n)} \Bigg(\max_{\substack{\tau : \angle(\tau, V_a) > \beta \\ \textrm{for } 1 \leq a \leq A}} \Big\|\sum_{T \in \T[\tau]} \chi_T\Big\|_{L^p(B_{\;\!\!\delta})}^p \Bigg),
\end{equation*}
where $A \in \N$ and $\mathrm{Gr}(k-1, n)$ is the Grassmannian manifold of all $(k-1)$-dimensional subspaces in $\R^n$. Here $\angle(\tau, V_a)$ denotes the infimum of the (unsigned) angles $\angle(v,v')$  over all pairs of non-zero vectors $v \in \tau$ and $v' \in V_a $. For $U\subseteq \R^n$ the \emph{$k$-broad norm over $U$} is then defined to be
\begin{equation*}
\Big\|\sum_{T \in \T} \chi_T\Big\|_{\mathrm{BL}^p_{k,A}(U)} := \Bigg(\sum_{\substack{B_{\;\!\!\delta} \in \mathcal{B}_{\delta} }}\frac{|B_{\;\!\!\delta}\cap U|}{|B_{\;\!\!\delta}|} \mu_{\T}(B_{\;\!\!\delta}) \Bigg)^{1/p}.
\end{equation*}
The $k$-broad norms are \emph{not} norms in any familiar sense, but they do satisfy weak analogues of various properties of $L^p$-norms. The basic properties of these objects are described in Section~\ref{k-broad norms section} below.

The main ingredient in the proof of Theorem~\ref{linear theorem} is the following estimate for $k$-broad norms.

\begin{theorem}\label{main theorem} Let   $p\ge1+\frac{1}{n-1}(\frac{n}{n-1})^{n-k}$. For all $\varepsilon>0$,  there is an $A\sim 1$ and a complexity $D\in \mathbb{N}$ such that
\begin{equation}\label{broad estimate}\tag{$\BL{k}^p$}
    \Big\|\sum_{T\in \T}\chi_{T}\Big\|_{\mathrm{BL}^p_{k,A}(\R^n)}\,\lesssim\, N^{1-1/p}\delta^{-(n-1 - n/p)-\varepsilon}\Big(\sum_{T \in \T} |T| \Big)^{1/p}\end{equation}
whenever $0<\delta<1$, $N\ge 1$, and $\mathbb{T}$ satisfies the $(D,N)$-polynomial Wolff axiom. 
\end{theorem}

The proof of Theorem~\ref{main theorem}, which is based on the polynomial partitioning method and closely follows the arguments of \cite{Guth2016, Guth2019, HR2018}, will be presented in Sections~\ref{k-broad norms section}--\ref{proof section}. 

The key feature which distinguishes the $k$-broad norm from its $L^p$ counterpart is that the former vanishes whenever the tubes of $\T$ cluster around a $(k-1)$-dimensional set (see Lemma~\ref{vanishing lemma} for a precise statement of this property). Owing to this special behaviour, the inequality~\eqref{broad estimate} is substantially weaker than~\eqref{kakeya}. Nevertheless, a mechanism introduced by Bourgain and Guth~\cite{BG2011} allows one to pass from $k$-broad to linear estimates, albeit under a rather stringent condition on the exponent. 

\begin{proposition}[Bourgain--Guth~\cite{BG2011}, Guth~\cite{Guth2019}]\label{Bourgain--Guth} 
Let  
$
 p \ge \frac{n-k+2}{n-k+1}$,  $\varepsilon > 0$, $A \sim 1$ and $D\in\mathbb{N}$. Suppose that
\begin{equation}\label{broad again}\tag{$\BL{k}^p$}
    \Big\|\sum_{T\in \T}\chi_{T}\Big\|_{\mathrm{BL}^p_{k,A}(\R^n)}\,\lesssim\, N^{1-1/p}\delta^{-(n-1 - n/p)-\varepsilon}\Big(\sum_{T \in \T} |T| \Big)^{1/p}
    \end{equation}
whenever $0<\delta<1$, $N\ge 1$, and $\mathbb{T}$ satisfies the $(D,N)$-polynomial Wolff axiom. Then
 \begin{equation}\label{linear again}\tag{$\mathrm{K}_p$}
\Big\| \sum_{T\in\mathbb{T}} \chi_T\Big\|_{L^p(\mathbb{R}^n)}\,\lesssim\, N^{1-1/p}\delta^{-(n-1 - n/p)-\varepsilon}\Big(\sum_{T \in \T} |T| \Big)^{1/p}
\end{equation}
whenever $0<\delta<1$, $N\ge 1$, and $\mathbb{T}$ satisfies the $(D,N)$-polynomial Wolff axiom.
 \end{proposition}
 
 Thus, combining Theorem~\ref{main theorem} and Proposition~\ref{Bourgain--Guth} yields Theorem~\ref{linear theorem}. In contrast with the range of Lebesgue exponents in Theorem~\ref{main theorem}, the range in which Proposition~\ref{Bourgain--Guth} applies shrinks as $k$ increases. The optimal compromise between the constraints in Theorem~\ref{main theorem} and Proposition~\ref{Bourgain--Guth} is given by \eqref{linear exponent}.

We end this section with a proof of Proposition~\ref{Bourgain--Guth}, which is a minor modification of the argument in~\cite{BG2011} (see also~\cite{Guth2019}).

\begin{proof}[Proof (of Proposition~\ref{Bourgain--Guth})] The proof is by an induction-on-scale argument. 

For the base case, fix $\delta\sim 1$ and let $\T$ be a family of $\delta$-tubes satisfying the $(D,N)$-polynomial Wolff axiom. If $\mathcal{B}$ is a cover of $\R^n$ by finitely-overlapping balls of radius 1, then
$$
 \Bigl\|\sum_{T\in \T}\chi_{T}\Big\|_{L^p(\R^n)}^p\,\leq\,     \sum_{B \in \mathcal{B}} \Bigl\|\sum_{\substack{T\in \T\\ T \cap B \neq \emptyset}}\chi_{T}\Big\|_{L^p(B)}^p \,\lesssim\,  \sum_{B \in \mathcal{B}} \#\{T \in \T : T \subset 3B\}^p.
$$
The polynomial Wolff axiom hypothesis implies that $\#\{T \in \T : T \subset 3B\} \lesssim N$ for each $B \in \mathcal{B}$ and so \eqref{linear again} follows from H\"older's inequality and the fact that any tube $T \in \T$ can belong to at most $O(1)$ of the balls $3B$. 

 Now let $\mathbf{C}$ be a fixed constant, chosen sufficiently large so as to satisfy the requirements of the forthcoming argument, and fix some small $\delta>0$.

\vspace{0.5em}
\noindent{\it Induction hypothesis:} Suppose the inequality
\begin{equation*}
    \Bigl\|\sum_{\widetilde{T}\in \widetilde{\T}}\chi_{\widetilde{T}}\Big\|_{L^p(\R^n)}\,\leq\,    \mathbf{C} N^{1-1/p} \tilde{\delta}^{-(n-1-n/p)-\varepsilon}\Big(\sum_{\widetilde{T} \in \widetilde{\T}} |\widetilde{T}|\Big)^{1/p}
\end{equation*}
holds whenever $\tilde{\delta}\in[2\delta ,1)$, $N\ge 1$,  and $\widetilde{\T}$ is a family of $\tilde{\delta}$-tubes satisfying the $(D,N)$-polynomial Wolff axiom.
\vspace{0.5em}

\noindent

Let $\T$ be a family of $\delta$-tubes satisfying the $(D,N)$-polynomial Wolff axiom. Fix a $\delta$-ball $B_{\;\!\!\delta} \in \mathcal{B}_{\delta}$ and subspaces $V_1, \dots, V_{\!A} \in \mathrm{Gr}(n,k-1)$ which obtain the minimum in the definition of $\mu_{\T}(B_{\;\!\!\delta})$; thus
\begin{equation*}
    \mu_{\T}(B_{\;\!\!\delta}) = \max_{\substack{\tau : \angle(\tau, V_a) > \beta \\ \textrm{for } 1 \leq a \leq A}} \Big\|\sum_{T \in \T[\tau]} \chi_T\Big\|_{L^p(B_{\;\!\!\delta})}^p .
\end{equation*}
Since $A \sim 1$ and $\#\{\tau : \angle(\tau, V_a) \leq \beta\} \sim \beta^{-(k-2)}$, by the triangle inequality followed by H\"older's inequality, 
\begin{align*}
    \int_{B_{\;\!\!\delta}}\big|\sum_{T \in \T} \chi_T\big|^p &\,\lesssim\, \int_{B_{\;\!\!\delta}}\big|\sum_{\substack{\tau : \angle(\tau, V_a) > \beta \\ \textrm{for } 1 \leq a \leq A}}\sum_{T \in \T[\tau]} \chi_T\big|^p + \sum_{a = 1}^A \int_{B_{\;\!\!\delta}}\big|\sum_{\tau : \angle(\tau, V_a) \leq \beta}\sum_{T \in \T[\tau]} \chi_T\big|^p \\
    &\,\lesssim\, \beta^{-(n-1)p} \mu_{\T}(B_{\;\!\!\delta}) + \beta^{-(k-2)(p-1)}\sum_{\tau}\int_{B_{\;\!\!\delta}}\big|\sum_{T \in \T[\tau]} \chi_T\big|^p.
\end{align*}
Summing the estimate over all the balls $B_{\;\!\!\delta} \in \mathcal{B}_{\delta}$, we find  that
\begin{equation*}
    \Big\|\sum_{T \in \T} \chi_T\Big\|_{L^p(\R^n)}^p \!\lesssim \beta^{-(n-1)p} \Big\|\sum_{T \in \T} \chi_T\Big\|_{\BL{k,A}^p(\R^n)}^p\! + \beta^{-(k-2)(p-1)}\sum_{\tau}\Big\|\sum_{T \in \T[\tau]} \!\!\chi_T\Big\|_{L^p(\R^n)}^p.
\end{equation*}

The first term on the right-hand side of the above display is estimated using the hypothesised broad estimate. For the second term, we apply a linear rescaling $L \colon \R^n \to \R^n$ so that
\begin{equation}\label{this}
    \Big\|\sum_{T \in \T[\tau]} \chi_T\Big\|^p_{L^p(\R^n)} = \beta^{n-1}\Big\|\sum_{T \in \T[\tau]} \chi_{L(T)}\Big\|^p_{L^p(\R^n)}
\end{equation}
where $\{L(T) : T \in \T[\tau]\}$ is essentially a collection of $\tilde{\delta}$-tubes with $\tilde{\delta} := \beta^{-1}\delta$. To be more precise, let $\omega \in S^{n-1}$ denote the centre of the cap $\tau$ and choose $L$ so that it fixes the 1-dimensional space spanned by $\omega$ and acts as a dilation by a factor of $\beta^{-1}$ on the orthogonal complement $\omega^{\perp}$. Writing $x \in \R^n$ as $x = (x',x_n)$ with $x' \in \omega^{\perp}$, for any $T \in \T[\tau]$ with $v := \mathrm{dir}(T)$ there exists some $u \in \R^n$ such that 
\begin{equation*}
    T \subseteq \big\{ x \in \R^n : |x' - u' - tv'| \lesssim \delta \textrm{ for some $|t| \leq 1$ and } |x_n - u_n| \leq 1/2 \big\},
\end{equation*}
Applying $L$ one obtains
\begin{equation*}
   L(T) \subseteq \big\{ y \in \R^n : |y' - \beta^{-1}u' - t\beta^{-1}v'| \lesssim \beta^{-1}\delta \textrm{ for some $|t| \leq 1$ and } |y_n - u_n| \leq 1/2 \big\}
\end{equation*}
and the right-hand side can be covered by a bounded number of $\tilde{\delta}$-tubes. Furthermore, the defining inequality of the polynomial Wolff axiom  is essentially invariant under this scaling and the family of $\tilde{\delta}$-tubes $L(T)$ continue to  satisfy the $(D, N)$-polynomial Wolff axiom.

Combining~\eqref{this} with the induction hypothesis we find that 
\begin{equation*}
 \Big\|\sum_{T \in \T[\tau]} \chi_T\Big\|^p_{L^p(\R^n)}   \,\lesssim\,   \beta^{n-1}\mathbf{C}^pN^{p-1}  (\beta^{-1}\delta)^{-(n-1)p+n-p\varepsilon}(\beta^{-1}\delta)^{n-1}\#\T[\tau].
\end{equation*}
Recalling  that $\sum_{\tau}\#\T[\tau]=\#\T$, by plugging the preceding estimate into our $L^p(\R^n)$-norm bound, 
\begin{equation*}
 \Big\|\sum_{T \in \T} \chi_T\Big\|_{L^p(\R^n)}^p  \leq\,  C \Big(\beta^{-(n-1)p} + \mathbf{C}^p\beta^{e(p,n,k)+p\varepsilon} \Big)N^{p-1} \delta^{-(n-1)p+n-p\varepsilon}\Big(\sum_{T \in \T} |T|\Big); 
\end{equation*}
here $C$ depends, amongst other things, on the implied constant in \eqref{broad again},  and 
\begin{equation*}
    e(p,n,k) := (n -k + 1)p - (n - k + 2).
\end{equation*}

By assumption,  $p\ge \frac{n-k+2}{n-k+1}$ and therefore $e(p,n,k) \ge0$. Consequently, $\beta$ may be chosen sufficiently small, depending only on the admissible parameters $n$, $p$ and~$\varepsilon$, so that
\begin{equation*}
    C \beta^{e(p,n,k)+p\varepsilon}  \,\leq\,  \frac{1}{2}.
\end{equation*}
Moreover, if $\mathbf{C}$ is chosen sufficiently large from the outset, it follows that 
\begin{equation*}
       \Big\|\sum_{T \in \T} \chi_T\Big\|_{L^p(\R^n)}^p \,\leq\,   \mathbf{C}^pN^{p-1} \delta^{-(n-1)p+n - p\varepsilon}\Big(\sum_{T \in \T} |T|\Big),
\end{equation*}
which closes the induction and completes the proof. 
\end{proof}




\section{Basic properties of the $k$-broad norms} \label{k-broad norms section}




\subsubsection*{Vanishing property} The proof of Theorem~\ref{main theorem} will involve analysing collections of tubes which enjoy certain tangency properties with respect to algebraic varieties.

\begin{definition} Given any collection of polynomials $P_1, \dots, P_{n-m} \colon \R^n \to \R$ the common zero set
\begin{equation*}
    Z(P_1, \dots, P_{n-m}) := \{x \in \R^n : P_1(x) = \cdots = P_{n-m}(x) = 0\}
\end{equation*}
will be referred to as a \emph{variety}.\footnote{Note that here, in contrast with much of the algebraic geometry literature, the ideal generated by the $P_j$ is not required to be irreducible.} Given a variety $\bZ = Z(P_1, \dots, P_{n-m})$, define its \emph{(maximum) degree} to be the number 
\begin{equation*}
\overline{\deg}\,\bZ := \max \{\deg P_1, \dots, \deg P_{n-m} \}.
\end{equation*}
\end{definition}
It will often be convenient to work with varieties which satisfy the additional property that
\begin{equation}\label{non singular variety}
\bigwedge_{j=1}^{n-m}\nabla P_j(z) \neq 0 \qquad \textrm{for all $z \in \bZ = Z(P_1, \dots, P_{n-m})$.}
\end{equation}
In this case the zero set forms a smooth $m$-dimensional submanifold of $\R^n$ with a (classical) tangent space $T_z\bZ$ at every point $z \in \bZ$. A variety $\bZ$ which satisfies~\eqref{non singular variety} is said to be an \emph{$m$-dimensional transverse complete intersection}.

\begin{definition}\label{tangent definition} Let $0 < \delta < r < 1$, $x_0 \in \R^n$ and $\bZ \subseteq \R^n$ be a transverse complete intersection. A $\delta$-tube $T \subset \R^n$ is \emph{tangent to $\bZ$ in $B(x_0, r)$} if 
\begin{enumerate}[i)]
\item $T \cap B(x_0, r) \cap N_{\;\!\!\delta}\bZ \neq \emptyset$ ;
\item If $x \in T$ and $z \in \bZ \cap B(x_0, 2r)$ satisfy $|z-x| \leq 4\delta$, then
\begin{equation*}
    \angle(\mathrm{dir}(T), T_z\bZ) \,\leq\,  c_{\tang} \frac{\delta}{r}. 
\end{equation*}
\end{enumerate}
\end{definition}

Here $0 < c_{\tang}$ is an admissible constant which is chosen small enough to ensure that, whenever i) and ii) hold, \begin{equation}\label{inclusion}T \cap B(x_0, 2 r) \subseteq N_{\;\!\!2 \delta}\bZ.\end{equation} The fact that such a choice is possible follows from a simple calculus exercise (see, for instance,~\cite[Proposition~9.2]{GHI} for details of an argument of this type).

The raison d'\^etre for the $k$-broad norms is the following lemma, which roughly states that the broad norms vanish if the tubes in $\T$ cluster around a low dimensional variety.

\begin{lemma}[Vanishing property]\label{vanishing lemma} Given $\varepsilon_{\;\!\!\circ} > 0$ and $0<\beta<1$ there exists some $0 < c < 1$ such that the following holds. Let $0 < \delta < c$, $r > \delta^{1- \varepsilon_{\;\!\!\circ}}$, $x_0 \in \R^n$ and $\bZ \subseteq \R^n$ be a transverse complete intersection of dimension at most $k-1$. Then
\begin{equation*}
    \Big\| \sum_{T \in \T} \chi_T \Big\|_{\mathrm{BL}^p_{k,A}(B(x_0, r))} = 0
\end{equation*}
whenever $\T$ is a family of $\delta$-tubes which are tangent to $\bZ$ in $B(x_0, r)$.\footnote{Here the parameter $\beta$ appears implicitly in the definition of the $k$-broad norm.}
 \end{lemma}
 
 \begin{proof} Fix $B_{\;\!\!\delta} \in \mathcal{B}_{\delta}$ with $B_{\;\!\!\delta} \cap B(x_0, r) \neq \emptyset$. Recalling the definition of the $k$-broad norm, it suffices to show that there exists some $V \in \mathrm{Gr}(k-1,n)$ such that 
 \begin{equation*}
     \max_{\tau  : \angle(\tau, V) > \beta} \int_{B_{\;\!\!\delta}}  \big| \sum_{T \in \T[\tau]} \chi_T \big|^p = 0.
 \end{equation*}
This would follow if $V$ has the property that
\begin{equation}\label{vanishing 1}
\textrm{if $T \in \T$ satisfies $T \cap B_{\;\!\!\delta} \neq \emptyset$, then $\angle(\mathrm{dir}(T), V) \leq \beta$.}
\end{equation}

Without loss of generality, one may assume there exists some $T_0 \in \T$ such that $T_0 \cap B_{\;\!\!\delta} \neq \emptyset$ (otherwise~\eqref{vanishing 1} vacuously holds for any choice of $(k-1)$-dimensional subspace). By the containment property resulting from the tangency hypothesis,
\begin{equation*}
  T_0 \cap  B_{\;\!\!\delta}  \subseteq T_0 \cap B(x_0,2r) \subseteq N_{\;\!\!2 \delta} \bZ
\end{equation*}
and therefore there exists some $z_0 \in \bZ$ such that $|z_0 - y_0| < 2 \delta$ for some $y_0 \in T_0 \cap B_{\;\!\!\delta}$. Let $V$ be a $(k-1)$-dimensional subspace containing $T_{z_0} \bZ$. Given any $T \in \T$, if $x \in T \cap B_{\;\!\!\delta}$ then $|x - z_0| < 4\delta$ and property ii) of the tangency hypothesis implies 
\begin{equation*}
     \angle(\mathrm{dir}(T), V) \,\lesssim\, \frac{\delta}{r} . 
\end{equation*}
Since $r > \delta^{1- \varepsilon_{\;\!\!\circ}}$, it follows that $ \angle(\mathrm{dir}(T), V) \leq \beta$ provided $\delta$ is sufficiently small depending only on $\varepsilon_{\;\!\!\circ}$ and $\beta$, which completes the proof.  
 \end{proof}



\subsubsection*{Triangle and logarithmic convexity inequalities} The $k$-broad norms satisfy weak  variants of certain key properties of $L^p$-norms.

\begin{lemma}[Finite subadditivity] Let $U_1, U_2 \subseteq \R^n$,  $1 \leq p < \infty$ and $A\in \mathbb{N}$. Then
\begin{equation*}
\Big\|\sum_{T \in \T} \chi_T\Big\|_{\mathrm{BL}^p_{k,A}(U_1 \cup U_2)}^p \,\leq\,  \Big\|\sum_{T \in \T} \chi_T\Big\|_{\mathrm{BL}^{p}_{k,A}(U_1)}^p + \Big\|\sum_{T \in \T} \chi_T\Big\|_{\mathrm{BL}^{p}_{k,A}(U_2)}^p
\end{equation*} 
whenever $\T$ is a family of $\delta$-tubes.
\end{lemma}

\begin{lemma}[Triangle inequality]\label{triangle inequality lemma} Let $U \subseteq \R^n$, $1 \leq p < \infty$ and $A\in\mathbb{N}$. Then
\begin{equation*}
\Big\|\sum_{T \in \T_1 \cup \T_2} \chi_T \Big\|_{\mathrm{BL}^p_{k,2A}(U)} \,\lesssim\, \Big\|\sum_{T \in \T_1} \chi_T\Big\|_{\mathrm{BL}^p_{k,A}(U)} + \Big\|\sum_{T \in \T_2} \chi_T\Big\|_{\mathrm{BL}^p_{k,A}(U)}
\end{equation*} 
whenever $\T_1$ and $\T_2$ are families of $\delta$-tubes.
\end{lemma}

\begin{lemma}[Logarithmic convexity]\label{logarithmic convexity inequality lemma} Let $U \subseteq \R^n$, $1 \leq p, p_0, p_1 < \infty$ and $A\in\mathbb{N}$. Suppose that $\theta\in[0,1]$ satisfies
\begin{equation*}
\frac{1}{p} = \frac{1-\theta}{p_0} + \frac{\theta}{p_1}.
\end{equation*}
Then
\begin{equation*}
\Big\|\sum_{T \in \T} \chi_T\Big\|_{\mathrm{BL}^p_{k,2A}(U)} \,\lesssim\, \Big\|\sum_{T \in \T} \chi_T\Big\|_{\mathrm{BL}^{p_0}_{k,A}(U)}^{1-\theta} \Big\|\sum_{T \in \T} \chi_T\Big\|_{\mathrm{BL}^{p_1}_{k,A}(U)}^{\theta}
\end{equation*} 
whenever $\T$ is a family of $\delta$-tubes.
\end{lemma}

These estimates are entirely elementary. The proofs are identical to those used to analyse broad norms in the context of the Fourier restriction problem~\cite{Guth2019}. It is remarked that the parameter $A$ appears in the definition of the $k$-broad norm to allow for these weak triangle and logarithmic convexity inequalities.

\subsubsection*{$k$-broad versus $k$-linear estimates}

Although not required for the proof of Theorem~\ref{linear theorem}, it is perhaps instructive to note the relationship between the $k$-broad norms and the multilinear expressions appearing in the work of Bennett--Carbery--Tao~\cite{BCT2006}.

\begin{proposition}\label{k-broad vrs k-linear} Let $\T$ be a collection of $\delta$-tubes in $\R^n$. Then
\begin{equation*}
\Big\|\sum_{T\in \T}\chi_{T}\Big\|_{\mathrm{BL}^p_{k,A}(\R^n)}\,\lesssim\, \Bigg(\sum_{\substack{(\tau_1, \dots, \tau_k) \\ \sim \,\beta^{k-1}\mathrm{\!\!-trans.}}}\Bigl\|\prod_{j=1}^{k}\Big(\sum_{T_{j}\in\T[\tau_{j}]}\chi_{N_{\;\!\!2\delta}T_{j}}\Big)^{1/k}\Bigr\|_{L^{p}(\R^n)}^p\Bigg)^{1/p}
\end{equation*}
where the sum is over all $k$-tuples $(\tau_1, \dots, \tau_k)$ of caps of diameter $\beta$ which are $\sim \beta^{k-1}$-transversal in the sense that $|\bigwedge_{j=1}^{k} \omega_j | \gtrsim \beta^{k-1}$ for all $\omega_j \in \tau_j$.
\end{proposition}

Thus, any $k$-linear inequality of the type featured in~\cite{BCT2006, Guth2010, BG2011} is stronger than the corresponding $k$-broad estimate (given that $\beta$ is admissible).

The proof of Proposition~\ref{k-broad vrs k-linear} is a simple exercise and is omitted (see~\cite{GHI} for similar results in the (more complicated) context of oscillatory integral operators). 


%




\section{Polynomial partitioning}\label{polynomial partitioning section}

In this section the algebraic and topological ingredients for the proof of Theorem~\ref{main theorem} are reviewed. In particular, the key polynomial partitioning theorem is recalled, which is adapted from~\cite{Guth2016, Guth2019} (see also~\cite{Wang2018}) 
and previously appeared explicitly in~\cite{HR2018}. 

 Given a polynomial $P \colon \R^n \to \R$ consider the collection  $\cell(P)$  
 of connected components of $\R^n \setminus Z(P)$. Each $O' \in \cell(P)$ is referred to as a \emph{cell} cut out by the variety $Z(P)$ and the cells are thought of as partitioning the ambient euclidean space into a finite collection of disjoint regions. 

In order to account for the choice of scale $\delta > 0$ appearing in the definition of the $\delta$-tubes, it will be useful to consider the family of \emph{$\delta$-shrunken cells} defined by
\begin{equation}\label{shrunken cells}
    \O := \big\{ O'\setminus N_{\;\!\!\delta}Z(P) : O' \in \cell(P) \big\}.
\end{equation}
An important consequence of this definition is the following simple observation:
\begin{quote}
    A $\delta$-tube $T$ can enter at most $\deg P + 1$ of the shrunken cells~$O \in \O$.
\end{quote}
Indeed, this is a simple and direct consequence of the fundamental theorem of algebra (or B\'ezout's theorem) applied to the core line of $T$.

\begin{theorem}[Guth~\cite{Guth2019}]\label{partitioning theorem} Fix $0 < \delta < r$, $x_0 \in \R^n$ and suppose $F \in L^1(\R^n)$ is non-negative and supported on $B(x_0,r) \cap N_{\;\!\!2\delta}\bZ$ where $\bZ$ is an $m$-dimensional transverse complete intersection with $\overline{\deg}\,\bZ \leq d$. At least one of the following cases holds:\\

\paragraph{\underline{Cellular case}} There exists a polynomial $P \colon \R^n \to \R$ of degree $O(d)$ with the following properties:
\begin{enumerate}[i)]
    \item $\#\cell(P) \sim d^m$ and each $O \in \cell(P)$ has diameter at most $r/2$.
    \item One may pass to a refinement of $\cell(P)$ such that if $\O$ is defined as in~\eqref{shrunken cells}, then 
    \begin{equation*}
        \int_{O} F \sim d^{-m}\int_{\R^n} F \qquad \textrm{for all $O \in \O$.}
    \end{equation*}
\end{enumerate}  
\paragraph{\underline{Algebraic case}} There exists an $(m-1)$-dimensional  transverse complete intersection $\bY$ of  degree at most $O(d)$ such that
    \begin{equation*}
        \int_{B(x_0,r) \cap N_{\;\!\!2\delta}\bZ} F \,\lesssim\,\log d \int_{B(x_0,r) \cap N_{\;\!\!\delta}\bY} F.
    \end{equation*}
\end{theorem}

This theorem is based on an earlier discrete partitioning result which played a central role in the resolution of the Erd\H{o}s distance conjecture~\cite{GK2015}. The proof is essentially topological, involving the polynomial ham sandwich theorem of Stone--Tukey~\cite{Stone1942}, which is itself a consequence of the Borsuk--Ulam theorem (see, for instance,~\cite{Matousek}), combined with a pigeonholing argument.

The theorem  is applied to $k$-broad norms by taking\begin{equation*}
   F= \sum_{B_{\;\!\!\delta} \in \cB_{\;\!\!\delta}} \mu_{\T}(B_{\;\!\!\delta}) \frac{1}{|B_{\;\!\!\delta}|} \chi_{B_{\;\!\!\delta}  }.
\end{equation*}
 \begin{itemize} 
 \item If the cellular case holds, then it follows that
 \begin{equation*}
    \Big\|\sum_{T\in \T}\chi_{T}\Big\|^p_{\mathrm{BL}^p_{k,A}(B(x_0,r) \cap N_{\;\!\!2\delta}\bZ)} \,\lesssim\, d^{-m}\Big\|\sum_{T\in \T}\chi_{T}\Big\|_{\BL{k,A}^p(O)} \ \textrm{for all $O \in \O$}
 \end{equation*}
 where $\O$ is the collection of cells produced by Theorem~\ref{partitioning theorem}.
 \item If the algebraic case holds, then it follows that
 \begin{equation*}
    \Big\|\sum_{T\in \T}\chi_{T}\Big\|^p_{\mathrm{BL}^p_{k,A}(B(x_0,r) \cap N_{\;\!\!2\delta}\bZ)} \,\lesssim\, 
     \log d\,\Big\|\sum_{T\in \T}\chi_{T}\Big\|^p_{\mathrm{BL}^p_{k,A}(B(x_0,r) \cap N_{\;\!\!\delta}\bY)}
\end{equation*}
where $\bY$ is the variety produced by Theorem~\ref{partitioning theorem}.
\end{itemize}




\section{Finding polynomial structure}\label{structure lemma section}

In this section, the recursive argument used to study the Fourier restriction problem in~\cite{HR2018} (which, in turn, is adapted from~\cite{Guth2019}) is reformulated so as to apply to the Kakeya problem. As in~\cite{HR2018}, the argument will be presented as two separate algorithms:
\begin{itemize}
    \item \texttt{[alg 1]} effects a dimensional reduction, essentially passing from an $m$-dimensional to an $(m-1)$-dimensional situation. 
    \item  \texttt{[alg 2]} consists of repeated application of the first algorithm to reduce to a minimal dimensional case. 
\end{itemize}

The final outcome is a method of decomposing any given $k$-broad norm into pieces which are either easily controlled or enjoy special algebraic structure. This decomposition applies to \emph{arbitrary} families of $\delta$-tubes. In the following section, we will specialise to the case where the tubes satisfy the polynomial Wolff axiom and use this additional information to prove Theorem~\ref{main theorem}. 

\subsection*{The first algorithm}  Throughout this section let $p \geq 1$ and $0< \varepsilon_{\;\!\!\circ} \ll \varepsilon \ll 1$ be fixed.\\

\paragraph{\underline{\texttt{Input}}} \texttt{[alg 1]} will take as its input:
\begin{itemize}
    \item A choice of small scale $0<\delta\ll 1$ and large scale $r_0\in[\delta^{1-\varepsilon_{\;\!\!\circ}},\delta^{\varepsilon_{\;\!\!\circ}}]$.
    \item A transverse complete intersection $\bZ$ of dimension $m\in\{2,\ldots,n\}$.
    \item A family  $\T$ of $\delta$-tubes which are tangent to $\bZ$ on a ball $B_{\;\!\!r_0}$ of radius $r_0$.
    \item A large integer $A \in \N$. 
\end{itemize}

%
%
%
\paragraph{\underline{\texttt{Output}}} \texttt{[alg 1]} will output a finite sequence of sets $(\sE_j)_{j=0}^J$, which are constructed via a recursive process. Each $\sE_j$ is referred to as an \emph{ensemble} and contains all the relevant information coming from the $j$th step of the algorithm. In particular, the ensemble $\sE_j$ consists of:
\begin{itemize}
    \item A word $\fh_j$ of length $j$ in the alphabet $\{\ta, \tc\}$, referred to as a {\it history}. 
    The  $\ta$ is an abbreviation of ``algebraic'' and~$\tc$ ``cellular''. The words $\fh_j$ are recursively defined by successively adjoining a single letter. Each $\fh_j$ records how the cells $O_j \in \O_j$ were constructed via repeated application of the polynomial partitioning theorem. 
    \item A large scale $r_j\in[\delta^{1-\varepsilon_{\;\!\!\circ}},\delta^{\varepsilon_{\;\!\!\circ}}]$. The $r_j$ will in fact be completely determined by the initial scales and the history $\fh_j$. In particular, 
let $\sigma_{k} \colon [0,1] \to [0,1]$ be given by
    \begin{equation*}
        \sigma_{k}(r) := \left\{ \begin{array}{ll}
        \frac{r}{2} &  \textrm {if the $k$th letter of $\fh_j$ is $\tc$} \\[6pt]
        r^{1+\varepsilon_{\;\!\!\circ}} & \textrm{if the $k$th letter of $\fh_j$ is $\ta$}
        \end{array}\right. 
    \end{equation*}
for each $1 \leq k \leq j$. With these definitions, 
\begin{equation*}
r_j := \sigma_{j} \circ \cdots \circ \sigma_{1}(r_0).
\end{equation*}
Note that each $\sigma_{k}$ is a decreasing function and 
\begin{equation}\label{radius bounds}
r_j \leq \delta^{\varepsilon_{\;\!\!\circ}(1+\varepsilon_{\;\!\!\circ})^{\#_{\sta}(j)}}   \quad \textrm{and} \quad r_j \leq 2^{-\#_{\stc}(j)}\delta^{\varepsilon_{\;\!\!\circ}}  
\end{equation}
  where $\#_{\bta}(j)$ and $\#_{\btc}(j)$ denote the number of occurrences of $\ta$ and $\tc$ in the history $\fh_j$, respectively. 
  \item A family of subsets $\O_j$ of $\R^n$ which will be referred to as \emph{cells}. Each cell $O_j \in \O_j$ is contained in $B_{\;\!\!r_0}$ and will have diameter at most $2r_j$. 
\item An assignment of a subfamily $\T[O_j]$ of $\delta$-tubes to each of the cells $O_j$. 
  \item A large integer $d \in \N$ which depends only on $\Deg \bZ$ and the admissible parameters $n$ and $\varepsilon$.
  \end{itemize}

Moreover, the components of the ensemble are defined so as to ensure that, for certain coefficients
\begin{equation*}
   C_{j}(d):= d^{\#_{\stc}(j)\varepsilon_{\;\!\!\circ}} d^{\#_{\sta}(j)(n+\varepsilon_{\;\!\!\circ})} 
\end{equation*}
and $A_j := 2^{-\#_{\sta}(j)}A \in \N$, the following properties hold:\\

\paragraph{\underline{Property I}} The  function $\sum_{T\in \T} \chi_T$ on $B_{\;\!\!r_0}$ can be compared with functions defined over the $\T[O_j]$:
\begin{equation} \tag*{$(\mathrm{I})_j$}
      \Big\|\sum_{T\in \T} \chi_T\Big\|_{\BL{k,A}^p(B_{\;\!\!r_0})}^p \leq \,\,C_{j}(d)\! \sum_{O_{j} \in \O_{j}} \Big\|\sum_{T\in \T[O_j]} \chi_T\Big\|_{\BL{k,A_j}^p(O_j)}^p .
\end{equation}
\\

\paragraph{\underline{Property II}} The tube families $\T[O_j]$ satisfy
\begin{equation}\tag*{$(\mathrm{II})_j$}
    \sum_{O_{j} \in \O_{j}} \# \T[O_j] \,\leq\, C_{j}(d)d^{\#_{\stc}(j)}  \#\T.
\end{equation}
\paragraph{\underline{Property III}} Furthermore, each individual $\T[O_j]$ satisfies
\begin{equation}\tag*{$(\mathrm{III})_j$}
     \#\T[O_{j}] \,\leq\, C_{j}(d)  d^{-\#_{\stc}(j)(m-1)}  \#\T.
\end{equation}

%
%
%

\subsection*{The initial step}

The initial ensemble $\sE_0$ is defined by taking:
\begin{itemize}
    \item $\fh := \emptyset$ to be the empty word;
    \item $r_0$ to be the large scale;
    \item $\O_0$ the collection consisting of the single ball $O_0 := B_{\;\!\!r_0}$;
    \item $\T[O_0] := \T$.
\end{itemize} 
All the desired properties then vacuously hold. 

At this point it is also convenient to fix some large $d \in \N$, to be determined later, which depends only on  $\Deg\bZ$ and the admissible parameters $n$ and $\varepsilon$. 

With these definitions, it is trivial to verify that Properties I, II and III hold.

%
%
%

\subsection*{The recursive step} Assume the ensembles $\sE_0, \dots, \sE_j$ have been constructed for some $j \in \N_0$ and that they all satisfy the desired properties. \\

%
%
%

\paragraph{\underline{\texttt{Stopping conditions}}} The algorithm has two stopping conditions which are labelled \texttt{[tiny]} and \texttt{[tang]}. 
\begin{itemize}
\item[\texttt{Stop:[tiny]}] The algorithm terminates if $r_j \leq \delta^{1-\varepsilon_{\;\!\!\circ}}$.
\end{itemize}

\begin{itemize}
\item[\texttt{Stop:[tang]}]Let $C_{\textrm{\texttt{tang}}}$ and $C_{\alg}$ be fixed constants, chosen large enough to satisfy the forthcoming requirements of the proof. The algorithm terminates if the inequalities
\begin{equation*}
    \sum_{O_j \in \O_j} \Big\|\sum_{T \in \T[O_j]} \chi_T\Big\|_{\BL{k,A_j}^p(O_j)}^p \,\leq\,  C_{\textrm{\texttt{tang}}} \log d  \sum_{S \in \Sc} \Big\|\sum_{T \in \T[S]} \chi_T\Big\|_{\BL{k,A_j/2}^p(B[S])}^p
\end{equation*}
and
\begin{align*}
   \sum_{S \in \Sc}  \#\T[S]  &\,\leq \,C_{\textrm{\texttt{tang}}}\delta^{-n\varepsilon_{\;\!\!\circ}}\!\!\sum_{O_j \in \O_j}  \#\T[O_j]; \\ 
  \nonumber
  \max_{S \in \Sc}  \#\T[S]  &\,\leq \,C_{\textrm{\texttt{tang}}}\max_{O_j \in \O_j}  \#\T[O_j]
\end{align*}
hold for some choice of:
\end{itemize}
\begin{itemize}
    \item $\Sc$ a collection of transverse complete intersections in $\R^n$ all of equal dimension $m-1$ and degree at most $C_{\alg}d$;
    \item An assignment of a subfamily $\T[S]$ of $\T$ and a $\max\{r_j^{1 +\varepsilon_{\;\!\!\circ}}, \delta^{1 -\varepsilon_{\;\!\!\circ}}\}$-ball $B[S]$ to each $S \in \Sc$ with the property that each $T \in \T[S]$ is tangent to $S$ in $B[S]$ in the sense of Definition~\ref{tangent definition}.
\end{itemize}

The stopping condition \texttt{[tang]} can be roughly interpreted as forcing the algorithm to terminate if one can pass to a lower dimensional situation. Indeed, by the inclusion property \eqref{inclusion}, the broad norm over $B[S]$ could instead be taken over a $2\delta$-neighbourhood of $S$.

If either of the above conditions hold, then the stopping time is defined to be $J := j$. Recalling~\eqref{radius bounds}, the stopping condition \texttt{[tiny]} implies that the algorithm must terminate after finitely many steps and, moreover, 
\begin{equation*}
  \#_{\bta}(J) \,\lesssim\, \varepsilon_{\;\!\!\circ}^{-1} \log(\varepsilon_{\;\!\!\circ}^{-1}) \quad \textrm{and} \quad  \#_{\btc}(J) \,\lesssim\, \log \delta^{-1}.
\end{equation*}
Note that there can be relatively few algebraic steps $\#_{\bta}(j)$ but there can many cellular steps $\#_{\btc}(j)$. The first of the above estimates can also be used to show that $C_{j}(d)\lesssim_{d,\varepsilon_{\;\!\!\circ}}d^{\#_{\stc}(j)\varepsilon_{\;\!\!\circ}}$ always holds. Furthermore, by choosing $A \geq 2^{\varepsilon_{\;\!\!\circ}^{-2}}$, say, one may ensure that the $A_j$ defined above are indeed integers.
\\

%
%
%

\paragraph{\underline{\texttt{Recursive step}}}

Suppose that neither stopping condition \texttt{[tiny]} nor \texttt{[tang]} is met. One proceeds to construct the ensemble $\sE_{j+1}$ as follows. 

Given $O_j \in \O_j$, apply the polynomial partitioning theorem with degree $d$ to 
\begin{equation*}
    \Big\|\sum_{T \in \T[O_j]} \chi_T\Big\|_{\BL{k,A_j}^p(O_j\cap N_{\;\!\!2\delta}\bZ)}^p = \Big\|\sum_{T \in \T[O_j]} \chi_T\Big\|_{\BL{k,A_j}^p(O_j)}^p.
\end{equation*}
 For each $O_j \in \O_j$ either the cellular or the algebraic case holds, as defined in Theorem~\ref{partitioning theorem}. Let $\O_{j,\cell}$ denote the subcollection of $\O_j$ consisting of all cells for which the cellular case holds and $\O_{j,\alg} := \O_j \setminus \O_{j,\cell}$. Thus, by $(\mathrm{I})_j$, one may bound $\|\sum_{T \in \T} \chi_T\|_{\BL{k,A}^p(B_{\;\!\!r_0})}^p$ by
\begin{equation*}
    C_{j}(d)  \Big[ \sum_{O_j \in \O_{j,\cell}} \Big\|\sum_{T \in \T[O_j]} \chi_T\Big\|_{\BL{k,A_j}^p(O_j)}^p + \sum_{O_j \in \O_{j,\alg}} \Big\|\sum_{T \in \T[O_j]} \chi_T\Big\|_{\BL{k,A_j}^p(O_j)}^p \Big];
\end{equation*}
the analysis is splits into two cases depending on which term in the above sum dominates.




\subsection*{$\blacktriangleright$ Cellular-dominant case} Suppose that the inequality
\begin{equation*}
     \sum_{O_j \in \O_{j,\alg}} \Big\|\sum_{T \in \T[O_j]} \chi_T\Big\|_{\BL{k,A_j}^p(O_j)}^p \leq \sum_{O_j \in \O_{j,\cell}} \Big\|\sum_{T \in \T[O_j]} \chi_T\Big\|_{\BL{k,A_j}^p(O_j)}^p 
\end{equation*}
holds so that 
\begin{equation}\label{cellular dominant case}
   \Big\|\sum_{T \in \T} \chi_T\Big\|_{\BL{k,A}^p(B_{\;\!\!r_0})}^p\leq \,2C_{j}(d)\!\! \sum_{O_j \in \O_{j,\cell}} \Big\|\sum_{T \in \T[O_j]} \chi_T\Big\|_{\BL{k,A_j}^p(O_j)}^p.
\end{equation}

\paragraph{\underline{Definition of $\sE_{j+1}$}} Define $\fh_{j+1}$ by adjoining the letter $\tc$ to the word $\fh_j$. Thus, it follows from the definitions that 
\begin{equation}\label{cellular word}
    r_{j+1} =  \tfrac{1}{2}r_j, \quad \#_{\btc}(j+1) = \#_{\btc}(j)+1 \quad \textrm{and} \quad \#_{\bta}(j+1) = \#_{\bta}(j).
\end{equation}

The next generation of cells $\O_{j+1}$ arise from the cellular decomposition guaranteed by Theorem~\ref{partitioning theorem}. Fix $O_j \in \O_{j, \cell}$ so that there exists some polynomial $P \colon \R^n \to \R$ of degree $O(d)$ with the following properties:
\begin{enumerate}[i)]
    \item $\#\cell(P) \sim d^m$ and each $O \in \cell(P)$ has diameter at most $2r_{j+1}$. 
    \item One may pass to a refinement of $\cell(P)$ such that if 
    \begin{equation*}
        \O_{j+1}(O_j) := \big\{ O\setminus N_{\;\!\!\delta}Z(P) : O \in \cell(P)\}
    \end{equation*}
    denotes the corresponding collection of $\delta$-shrunken cells, then 
    \begin{equation*}
    \Big\|\sum_{T \in \T[O_j]} \chi_T\Big\|_{\BL{k,A_j}^p(O_j)}^p \,\lesssim\, d^{m}\Big\|\sum_{T \in \T[O_j]} \chi_T\Big\|_{\BL{k,A_j}^p(O_{j+1})}^p
\end{equation*}
for all $O_{j+1} \in \O_{j+1}(O_j)$.
    \end{enumerate}
 Given $O_{j+1} \in \O_{j+1}(O_j)$, define
 \begin{equation*}
 \T[O_{j+1}] := \big\{ T \in \T[O_j] : T \cap O_{j+1} \neq \emptyset \big\}.
\end{equation*}
 Recall that, by the fundamental theorem of algebra (or B\'ezout's theorem), any $\delta$-tube~$T$ can enter at most $O(d)$ cells $O_{j+1} \in \O_{j+1}(O_j)$ and, consequently, 
\begin{equation}\label{cellular 1}
    \sum_{O_{j+1} \in \O_{j+1}(O_j)} \#\T[O_{j+1}] \,\lesssim\, d \cdot \#\T[O_j].
\end{equation}
By the pigeonhole principle, one may pass to a refinement of $\O_{j+1}(O_j)$ such that
 \begin{equation}\label{cellular 2}
    \#\T[O_{j+1}] \,\lesssim\, d^{-(m-1)} \#\T[O_j] \qquad \textrm{for all $O_{j+1} \in \O_{j+1}(O_j)$.}
\end{equation}
 Finally, define 
 \begin{equation*}
 \O_{j+1} := \bigcup_{O_j \in \O_{j, \cell}} \O_{j+1}(O_j).
\end{equation*}
This completes the construction of $\sE_{j+1}$ and 
it remains to check that the new ensemble satisfies the desired properties. In view of this, it is useful to note that
\begin{equation}\label{cellular coefficients}
C_{j}(d) = d^{-\varepsilon_{\;\!\!\circ}} C_{j+1}(d) \quad \textrm{and}\quad A_j = A_{j+1},
\end{equation}
which follows immediately from~\eqref{cellular word} and the definition of the $C_{j}(d)$ and $A_j$. \\

%
%
%
\paragraph{\underline{Property I}} Fix $O_j \in \O_{j,\cell}$ and observe that $\#\O_{j+1}(O_j) \sim d^m$ and
\begin{equation*}
    \Big\|\sum_{T \in \T[O_j]} \chi_T\Big\|_{\BL{k,A_j}^p(O_j)}^p  \,\lesssim\, d^{m}\Big\|\sum_{T \in \T[O_{j+1}]} \chi_T\Big\|_{\BL{k,A_j}^p(O_{j+1})}^p
\end{equation*}
for all $O_{j+1} \in \O_{j+1}(O_j)$. Averaging,
\begin{equation*}
    \Big\|\sum_{T \in \T[O_j]} \chi_T\Big\|_{\BL{k,A_j}^p(O_j)}^p \,\lesssim \sum_{O_{j+1} \in \O_{j+1}(O_j)}  \Big\|\sum_{T \in \T[O_{j+1}]} \chi_T\Big\|_{\BL{k,A_j}^p(O_{j+1})}^p
\end{equation*}
and, recalling~\eqref{cellular dominant case} and~\eqref{cellular coefficients}, one deduces that
\begin{equation*}
    \Big\|\sum_{T \in \T} \chi_T\Big\|_{\BL{k,A}^p(B_{\;\!\!r_0})}^p \leq \,\,C d^{-\varepsilon_{\;\!\!\circ}} C_{j+1}(d)\!\! \sum_{O_{j+1} \in \O_{j+1}}  \Big\|\sum_{T \in \T[O_{j+1}]} \chi_T\Big\|_{\BL{k,A_{j+1}}^p(O_{j+1})}^p.
\end{equation*}
Provided $d$ is chosen large enough so as to ensure that the additional $d^{-\varepsilon_{\;\!\!\circ}}$ factor absorbs the unwanted constant $C$, one deduces $(\mathrm{I})_{j+1}$. This should be compared
with the approach of Solymosi and Tao to polynomial partitioning \cite{TS2012}.
\\

%
%
%
\paragraph{\underline{Property II}} By the construction, 
\begin{align*}
    \sum_{O_{j+1} \in \O_{j+1}} \#\T[O_{j+1}] &\,=\, \sum_{O_j \in \O_j}  \sum_{O_{j+1} \in \O_{j+1}(O_j)} \#\T[O_{j+1}]  \\
    &\,\lesssim\, d \sum_{O_j \in \O_j} \#\T[O_j], 
\end{align*}
where the inequality follows from a term-wise application of~\eqref{cellular 1}. Thus, $(\mathrm{II})_j$,~\eqref{cellular word} and~\eqref{cellular coefficients} imply that
\begin{equation*}
    \sum_{O_{j+1} \in \O_{j+1}} \#\T[O_{j+1}] \,\lesssim\, d^{-\varepsilon_{\;\!\!\circ}} C_{j+1}(d)  d^{\#_{\stc}(j+1)} \#\T.
\end{equation*}
Provided $d$ is chosen sufficiently large, one deduces $(\mathrm{II})_{j+1}$.\\

%
%
%
\paragraph{\underline{Property III}} Fix $O_{j+1} \in \O_{j+1}(O_j)$ and recall from~\eqref{cellular 2} that
\begin{equation*}
    \#\T[O_{j+1}] \,\lesssim\, d^{-(m-1)}\#\T[O_{j}].
\end{equation*}
Thus, $(\mathrm{III})_j$,~\eqref{cellular word} and~\eqref{cellular coefficients} imply that
\begin{equation*}
    \#\T[O_{j+1}]\,\lesssim\, d^{-\varepsilon_{\;\!\!\circ}} C_{j+1}(d) d^{-\#_{\stc}(j+1)(m-1)} \#\T[O_{j}].
\end{equation*}
Provided $d$ is chosen sufficiently large as before, one deduces $(\mathrm{III})_{j+1}$.




\subsection*{$\blacktriangleright$ Algebraic-dominant case} Suppose the hypothesis of the cellular-dominant case fails so that
\begin{equation}\label{algebraic dominant case}
   \Big\|\sum_{T \in \T} \chi_T\Big\|_{\BL{k,A}^p(B_{\;\!\!r_0})}^p \leq \,\,2C_{j}(d)\!\!  \sum_{O_j \in \O_{j,\alg}} \Big\|\sum_{T \in \T[O_j]} \chi_T\Big\|_{\BL{k,A_j}^p(O_j)}^p.
\end{equation} 
Each cell in $\O_{j,\alg}$ satisfies the condition of the algebraic case of Theorem~\ref{partitioning theorem}; this information is used to construct the $(j+1)$-generation ensemble. \\

\paragraph{\underline{Definition of $\sE_{j+1}$}} Define $\fh_{j+1}$ by adjoining the letter $\ta$ to the word $\fh_j$. Thus, it follows from the definitions that 
\begin{equation}\label{algebraic word}
    r_{j+1} =  r_j^{1+\varepsilon_{\;\!\!\circ}}, \quad \#_{\btc}(j+1) = \#_{\btc}(j) \quad \textrm{and} \quad \#_{\bta}(j+1) = \#_{\bta}(j)+1.
\end{equation} 
 The next generation of cells is constructed from the varieties  which arise from the algebraic case in Theorem~\ref{partitioning theorem}. Fix $O_j \in \O_{j,\alg}$ so that there exists a transverse complete intersection $\bY_{\!j}$ of dimension $m-1$ and $\Deg \bY_{\!j} \leq C_{\mathrm{alg}} d$ such that
\begin{equation*}
    \Big\|\sum_{T \in \T[O_j]} \chi_T\Big\|_{\BL{k,A_j}^p(O_j)}^p \,\lesssim\,  \log d\,\Big\|\sum_{T \in \T[O_j]} \chi_T\Big\|_{\BL{k,A_{j}}^p(O_j \cap N_{\;\!\!\delta}\bY_{\!j})}^p.
\end{equation*}
Let $\cB(O_j)$ be a  cover of $O_j \cap N_{\;\!\!\delta}\bY_{\!j}$ consisting of finitely-overlapping balls of radius $\max\{r_{j+1},\delta^{1 -\varepsilon_{\;\!\!\circ}}\}$. For each $B \in \cB(O_j)$ let $\T_{\!B}$ denote the family of $T \in \T[O_j]$ for which $T \cap B \cap N_{\;\!\!\delta}\bY_{\!j}  \neq \emptyset$. This set is partitioned into the subsets
\begin{equation*}
    \T_{\!B, \tang} := \big\{ T \in \T_{\!B} : \textrm{$T$ is tangent to $\bY_{\!j}$ on $B$} \big\}, \quad \T_{\!B, \trans} := \T_{\!B} \setminus \T_{\!B, \tang};
\end{equation*}
here the notion of tangency is that given in Definition~\ref{tangent definition}.

 By hypothesis, \texttt{[tang]} fails and, consequently, one may deduce that
\begin{equation}\label{algebraic 1}
    \sum_{O_j\in \O_{j,\mathrm{alg}}}\Big\|\sum_{T \in \T[O_j]} \chi_T\Big\|_{\BL{k,A_j}^p(O_j)}^p \lesssim\,\, \log d\!\! \sum_{\substack{O_j \in \O_{j,\mathrm{alg}}\\B \in \cB(O_j)}} \Big\|\sum_{T \in \T_{\!B, \trans}} \chi_T\Big\|_{\BL{k,A_{j+1}}^p(B_j)}^p
\end{equation}
where, for notational convenience, $B_j := B \cap N_{\;\!\!\delta}\bY_{\!j}$. Indeed, provided $C_{\textrm{\texttt{tang}}} > 0$ is sufficiently large,
\begin{align}\label{tangent bounds} \nonumber
   \sum_{O_j \in \O_{j,\alg}} \sum_{B \in \cB(O_j)}  \#\T_{\!B, \tang}  &\,\leq\,\, C_{\textrm{\texttt{tang}}}\delta^{-n\varepsilon_{\;\!\!\circ}}\!\!\sum_{O_j \in \O_j}  \#\T[O_j]; \\ 
  \max_{O_j \in \O_{j,\alg}} \max_{B \in \cB(O_j)}  \#\T_{\!B, \tang}  &\,\leq\,\, \max_{O_j \in \O_j}  \#\T[O_j].
\end{align}
Consequently, the failure of the stopping condition \texttt{[tang]} forces
\begin{equation*}
    \log d\!\!\sum_{O_j \in \O_j} \sum_{B \in \cB(O_j)} \Big\|\sum_{T \in \T_{\!B, \tang}} \!\!\!\!\chi_T\Big\|_{\BL{k,A_{j+1}}^p(B)}^p < \frac{1}{C_{\mathrm{tang}}}  \sum_{O_j \in \O_j} \Big\|\sum_{T \in \T[O_j]} \!\!\!\chi_T\Big\|_{\BL{k,A_j}^p(O_j)}^p 
\end{equation*}
(since the estimates in~\eqref{tangent bounds} show all other conditions for \texttt{[tang]} are met for $\mathcal{S}$, $\T[S]$ and $B[S]$ appropriately defined). On the other hand, by the triangle inequality for broad norms (Lemma~\ref{triangle inequality lemma}), using the fact that $A_{j+1} = A_j/2$, the left-hand side of~\eqref{algebraic 1} is dominated by
\begin{equation*}
 \log d\!\! \sum_{O_j \in \O_{j,\mathrm{alg}}} \sum_{B \in \cB(O_j)} \Big[\big\|\sum_{T \in \T_{\!B, \tang}}\!\!\chi_T\big\|_{\BL{k,A_{j+1}}^p(B_j)} + \big\|\sum_{T \in \T_{\!B, \trans}}\!\!\chi_T\big\|_{\BL{k,A_{j+1}}^p(B_j)}^p\Big].
\end{equation*}
For a suitable choice of constant $C_{\mathrm{tang}}$, combining the information in the two previous displays yields~\eqref{algebraic 1}.

For $O_j \in \O_{j,\alg}$ define 
\begin{equation*}
    \O_{j+1}(O_j) := \big\{ B \cap N_{\;\!\!\delta}\bY_{\!j} : B \in \cB(O_j) \big\}
\end{equation*}
and let $\T[O_{j+1}] := \T_{\!B, \trans}$ for $O_{j+1} = B \cap N_{\;\!\!\delta}\bY_{\!j} \in \O_{j+1}(O_j)$. The collection of cells $\O_{j+1}$ is then given by
\begin{equation*}
    \O_{j+1} := \bigcup_{O_j \in \O_{j, \alg}} \O_{j+1}(O_j). 
\end{equation*}
It remains to verify that the ensemble $\sE_{j+1}$ satisfies the desired properties. In view of this, it is useful to note that
\begin{equation}\label{algebriac coefficients}
C_{j}(d) = d^{-(n+\varepsilon_{\;\!\!\circ})}C_{j+1}(d),
\end{equation}
which follows directly from the definition of $C_{j}(d)$ and \eqref{algebraic word}.\\

%
%
%
\paragraph{\underline{Property I}} By combining~\eqref{algebraic 1} together with the various definitions one obtains 
\begin{equation*}
  \sum_{O_{j} \in \O_{j,\mathrm{alg}}} \Big\|\sum_{T \in \T[O_j]} \chi_T\Big\|_{\BL{k,A_j}^p(O_{j})}^p \,\lesssim\,\, \log d\!\!\! \sum_{O_{j+1} \in \O_{j+1}} \Big\|\sum_{T \in \T[O_{j+1}]} \chi_T\Big\|_{\BL{k,A_{j+1}}^p(O_{j+1})}^p.
\end{equation*}
Recalling~\eqref{algebraic dominant case} and~\eqref{algebriac coefficients}, if $c(d) := Cd^{-(n+\varepsilon_{\;\!\!\circ})}\log d$ for an appropriate choice of admissible constant $C$, then
\begin{equation*}
\Big\|\sum_{T \in \T} \chi_T\Big\|_{\BL{k,A}^p(B_{\;\!\!r_0})}^p \,\leq \,\,c(d)C_{j+1}(d)\!\!  \sum_{O_{j+1} \in \O_{j+1}}\Big\|\sum_{T \in \T[O_{j+1}]} \chi_T\Big\|_{\BL{k,A_{j+1}}^p(O_{j+1})}^p.
\end{equation*}
 Provided $d$ is sufficiently large, $c(d) \leq 1$ and one thereby deduces $(\mathrm{I})_{j+1}$. 
\\
\paragraph{\underline{Property II}} Fix $O_j \in \O_{j,\mathrm{alg}}$ and note that
\begin{equation}\label{algbraic case property II}
    \sum_{O_{j+1} \in \O_{j+1}(O_j)} \# \T[O_{j+1}] = \sum_{B \in \cB(O_j)}  \#\T_{\!B, \trans} 
\end{equation}
by the definition of $\T[O_{j+1}]$. To estimate the latter sum one may invoke the following algebraic-geometric result of Guth, which appears in Lemma 5.7 of~\cite{Guth2019}. 

\begin{lemma}[\cite{Guth2019}]\label{transversal intersection lemma}
Suppose $T$ is an infinite cylinder in $\R^n$ of radius $\delta$ and central axis $\ell$ and $\bY$ is a transverse complete intersection. For $\alpha > 0$ let
\begin{equation*}
    \bY_{\!> \alpha} := \big\{\, y \in \bY \,:\, \angle(T_y\bY, \ell) > \alpha \,\big\}.
\end{equation*}
The set $\bY_{\!> \alpha} \cap T$ is contained in a union of $O\big((\Deg \bY)^n\big)$ balls of radius $\delta\alpha^{-1}$.
\end{lemma}

Since $T \cap B \cap N_{\;\!\!\delta}\bY \neq \emptyset$ by the definition of $\T_{\!B}$, a tube $T \in \T_{\!B}$ belongs to $\T_{\!B,\trans}$ if and only if the angle condition ii) from Definition~\ref{tangent definition} fails to be satisfied. Thus, given any $T \in \bigcup_{B \in \cB} \T_{\!B, \trans}$, it follows from the definitions that 
\begin{equation*}
    \angle(\mathrm{dir}(T), T_y \bY) \,\gtrsim\, \frac{\delta}{r_{j+1}}
\end{equation*}
for some $y \in \bY\cap 2B$ with $|y - x| \lesssim \delta$ for some $x\in T$. This implies that 
\begin{equation*}
    N_{\;\!\!C\delta}T \cap 2B \cap \bY_{\!> \alpha_{j+1}} \neq \emptyset
\end{equation*}
where $\alpha_{j+1} \sim \delta/r_{j+1}$. Consequently, by Lemma~\ref{transversal intersection lemma}, any $T \in \bigcup_{B \in \cB(O_j)} \T_{\!B, \trans}$ lies in at most $O(d^n)$ of the sets~$\T_{\!B}$ and so
\begin{equation*}
    \sum_{B \in \cB(O_j)}  \#\T_{\!B, \trans} \,\lesssim\, d^n \#\T[O_j].
\end{equation*}
Combining this inequality with~\eqref{algbraic case property II} and summing over all $O_j \in \O_{j,\mathrm{alg}}$,
\begin{equation*}
    \sum_{O_{j+1} \in \O_{j+1}}  \# \T[O_{j+1}] \,\lesssim\, d^n \sum_{O_{j} \in \O_{j}}\#\T[O_j].
\end{equation*}
Applying $(\mathrm{II})_j$,~\eqref{algebraic word} and~\eqref{algebriac coefficients}, one concludes that
\begin{equation*}
    \sum_{O_{j+1} \in \O_{j+1}} \#\T [O_{j+1}] \,\lesssim\, d^{-\varepsilon_{\;\!\!\circ}}C_{j+1}(d)d^{\#_{\stc}(j+1)} \#\T.
\end{equation*}
Provided $d$ is chosen to be sufficiently large to absorb the implicit constant, one deduces $(\mathrm{II})_{j+1}$. 
\\

\paragraph{\underline{Property III}} Fix $O_j \in \O_{j,\mathrm{alg}}$ and $O_{j+1} \in \O_{j+1}(O_j)$. By definition, $\T[O_{j+1}] \subseteq \T[O_{j}]$ and so \begin{equation*}
   \#\T[O_{j+1}] \,\leq\, \#\T[O_{j}] \,\leq\, C_{j+1}(d) d^{-\#_{\stc}(j+1)(m-1)}\#\T,
\end{equation*}
by $(\mathrm{III})_j$ and~\eqref{algebraic word}.

\subsection*{The second algorithm} The algorithm \texttt{[alg 1]} is now applied repeatedly in order to arrive at a final decomposition of the $k$-broad norm. This process forms part of a second algorithm, referred to as \texttt{[alg 2]}.

Throughout this section  let $p_{\ell}$, with $k \leq \ell \leq n$, denote some choice of Lebesgue exponents satisfying $p_k \geq p_{k+1} \geq \dots \geq p_n =: p \geq 1.$ The numbers $0 \leq  \Theta_{\ell} \leq 1$ are then defined in terms of the $p_{\ell}$ by
\begin{equation*}
   \Theta_\ell :=\Big(1-\frac{1}{p_\ell}\Big)^{-1}\Big(1-\frac{1}{p}\Big)
\end{equation*}
so that $\Theta_n=1$. Also fix $0 < \varepsilon_{\;\!\!\circ}  \ll \varepsilon \ll 1$ as in the previous section. 

There are two stages to \texttt{[alg 2]}, which can roughly be described as follows:
\begin{itemize}
    \item \textbf{The recursive stage}: $\sum_{T \in \T}\chi_T$ is repeatedly decomposed into pieces with favourable tangency properties with respect to varieties of progressively lower dimension.  
    \item \textbf{The final stage}: $\sum_{T \in \T}\chi_T$ is further decomposed into very small scale pieces.
\end{itemize}
To begin, the recursive stage of \texttt{[alg 2]} is described. \\

\paragraph{\underline{\texttt{Input}}} \texttt{[alg 2]} will take as its input:
\begin{itemize}
    \item A choice of small scale $0 < \delta  \ll 1$.
    \item A large integer $A \in \N$. 
    \item A family of $\delta$-tubes $\T$ which are non-degenerate in the sense that
\begin{equation}\label{non degeneracy hypothesis}
\Big\|\sum_{T \in \T} \chi_T\Big\|_{\BL{k,A}^{p}(\R^n)} \neq 0.
\end{equation}
\end{itemize}
Note that the process applies to essentially arbitrary families of $\delta$-tubes (in particular, the polynomial Wolff axiom hypothesis does not appear at this stage).\\

\paragraph{\underline{\texttt{Output}}} The $(n+1-\ell)$th step of the recursion will produce:
\begin{itemize}
    \item An $(n+1-\ell)$-tuple of:
    \begin{itemize}
        \item scales $\vec{\delta}_{\ell} = (\delta_n, \dots, \delta_{\ell})$ satisfying $\delta^{\varepsilon_{\;\!\!\circ}} = \delta_n >  \dots > \delta_{\ell} \ge \delta^{1 - \varepsilon_{\;\!\!\circ}}$;
        \item large and (in general) non-admissible parameters $\vec{D}_{\ell} = (D_n, \dots, D_{\ell})$;
        \item integers $\vec{A} = (A_n, \dots, A_{\ell})$ satisfying $A = A_n > A_{n-1} > \dots > A_{\ell}$.
    \end{itemize}
    Each of these $(n+1-\ell)$-tuples is formed by adjoining a component to the corresponding $(n-\ell)$-tuple from the previous stage.
    \item A family $\vec{\Sc}_{\ell}$ of $(n+1-\ell)$-tuples of transverse complete intersections $\vec{S}_{\ell} = (S_n, \dots, S_{\ell})$ satisfying $\dim S_i = i$ and $\Deg S_i = O(1)$ for $\ell \leq i \leq n$. 
    \item An assignment of a $\delta_{\ell}$-ball $B[\vec{S}_{\ell}]$ and a subfamily $\T[\vec{S}_{\ell}]$ of $\delta$-tubes to each $\vec{S}_{\ell} \in \vec{\Sc}_{\ell}$ with the property that the tubes $T \in \T[\vec{S}_{\ell}]$ are tangent to $S_{\ell}$ in $B[\vec{S}_{\ell}]$ (here $S_{\ell}$ is the final component of $\vec{S}_{\ell}$).
\end{itemize}
This data is chosen so that the following properties hold:

\begin{notation} Throughout this section a large number of harmless $\delta^{-\varepsilon_{\;\!\!\circ}}$-factors appear in the inequalities. For notational convenience, given $A, B \geq 0$ let $A \lessapprox B$ or $B \gtrapprox A$ denote $A \lesssim \delta^{-c\varepsilon_{\;\!\!\circ}} B$ for some $c>0$ depending only on $n$ and $p$.
\end{notation}

\paragraph{\underline{Property 1}} The inequality
\begin{equation}\label{l step}
      \Big\|\sum_{T \in \T} \chi_T\Big\|_{\BL{k,A}^p(\R^n)} \lessapprox C(\vec{D}_{\ell}; \vec{\delta}_{\ell}) [\delta^n \#\T]^{1-\Theta_{\ell}} \Big( \sum_{\vec{S}_{\ell} \in \vec{\Sc}_{\ell}} \Big\|\sum_{T \in \T[\vec{S}_{\ell}]} \!\!\!\chi_T\Big\|_{\BL{k,A_{\ell}}^{p_{\ell}}(B[\vec{S}_{\ell}])}^{p_{\ell}}\Big)^{\frac{\Theta_{\ell}}{p_{\ell}}}
\end{equation}
holds for 
\begin{equation*}
    C(\vec{D}_{\ell}; \vec{\delta}_{\ell}) := \prod_{i=\ell}^{n-1} \Big(\frac{\delta_i}{\delta}\Big)^{\Theta_{i+1}-\Theta_{i}}D_i^{(1 + \varepsilon_{\;\!\!\circ})(\Theta_{i+1} - \Theta_{\ell})}.
\end{equation*}
\paragraph{\underline{Property 2}}  For $\ell \leq n-1$, the inequality
\begin{equation*}
    \sum_{\vec{S}_{\ell} \in \vec{\Sc}_{\ell}} \# \T[\vec{S}_{\ell}] \,\lessapprox\,\,D_{\ell}^{1 + \varepsilon_{\;\!\!\circ}} \!\!\!\!\sum_{\vec{S}_{\ell+1} \in \vec{\Sc}_{\ell+1}} \# \T[\vec{S}_{\ell+1}] 
\end{equation*}
holds.\\

\paragraph{\underline{Property 3}}    For $\ell \leq n-1$, the inequality
\begin{equation*}
    \max_{\vec{S}_{\ell} \in \vec{\Sc}_{\ell}} \# \T[\vec{S}_{\ell}] \,\lessapprox\,\,D_{\ell}^{-\ell + \varepsilon_{\;\!\!\circ}} \!\!\!\!\max_{\vec{S}_{\ell+1} \in \vec{\Sc}_{\ell+1}} \# \T[\vec{S}_{\ell+1}]
\end{equation*}
holds.\\

By the inclusion property \eqref{inclusion}, the broad norms over $B[\vec{S}_{\ell}]$ on the right-hand side of \eqref{l step} could be replaced by broad norms over $2\delta$-neighbourhoods of $S_{\ell}$.\\

\paragraph{\underline{\texttt{First step}}} Vacuously, the tubes belonging to $\T$ are tangent to the $n$-dimensional variety $\R^n$. Let $\mathcal{B}_{\circ}$ denote a collection of finitely-overlapping balls of radius $\delta^{\varepsilon_{\circ}}$ which cover $\bigcup_{T \in \T} T$ and define
\begin{itemize}
        \item $\delta_n := \delta^{\varepsilon_{\;\!\!\circ}}$; $D_n := 1$ and $A_n:= A$;
        \item $\Sc_n$ is the collection consisting of repeated copies of the 1-tuple $(\R^n)$, with one copy for each ball in $\mathcal{B}_{\circ}$;
        \item For each $\vec{S}_n \in \Sc_n$ assign a ball $B[\vec{S}_n] \in \mathcal{B}_{\circ}$ and let 
        \begin{equation*}
            \T[\vec{S}_n] := \big\{T \in \T : T \cap B[\vec{S_n}] \neq \emptyset\big\}.
        \end{equation*}
\end{itemize}
By a straightforward orthogonality argument (identical to that used to establish the base case in the proof of Proposition \ref{Bourgain--Guth}), Property 1 can be shown to hold with $C(\vec{D}_{n}; \vec{\delta}_{n})=1$ and $\Theta_n=1$.\\

\paragraph{\underline{\texttt{($n+2-\ell$)th step}}} Let $\ell \geq 1$ and suppose that the recursive algorithm has ran through $n+1-\ell$ steps. Since each family $\T[\vec{S}_{\ell}]$ consists of $\delta$-tubes which are tangent to $S_{\ell}$ on $B[\vec{S}_{\ell}]$, one may apply \texttt{[alg 1]} to bound the $k$-broad norm 
\begin{equation*}
    \Big\|\sum_{T \in \T[\vec{S}_{\ell}]} \chi_T\Big\|_{\BL{k,A_{\ell}}^{p_{\ell}}(B[\vec{S}_{\ell}])}.
\end{equation*} 
One of two things can happen: either \texttt{[alg 1]} terminates due to the stopping condition \texttt{[tiny]} or it terminates due to the stopping condition \texttt{[tang]}. The current recursive process terminates if the contributions from terms of the former type dominate:\\

\paragraph{\underline{\texttt{Stopping condition}}} The recursive stage of \texttt{[alg 2]} has a single stopping condition, which is denoted by \texttt{[tiny-dom]}. 
\begin{itemize}
    \item[\texttt{Stop:[tiny-dom]}] Suppose that the inequality
  \begin{equation}\label{tiny case 1}
      \sum_{\vec{S}_{\ell} \in \vec{\Sc}_{\ell}} \Big\|\sum_{T \in \T[\vec{S}_{\ell}]} \chi_T\Big\|_{\BL{k,A_{\ell}}^{p_{\ell}}(B[\vec{S}_{\ell}])}^{p_{\ell}} \,\leq\, \frac{1}{2}  \sum_{\vec{S}_{\ell} \in \vec{\Sc}_{\ell,\textrm{\texttt{tiny}}}}  \Big\|\sum_{T \in \T[\vec{S}_\ell]} \chi_T \Big\|_{\BL{k,A_{\ell}}^{p_{\ell}}(B[\vec{S}_{\ell}])}^{p_{\ell}}
\end{equation}
holds, where the right-hand summation is restricted to those $S_{\ell} \in \vec{\Sc}_{\ell}$ for which \texttt{[alg 1]} terminates owing to the stopping condition \texttt{[tiny]}. Then \texttt{[alg 2]} terminates.
\end{itemize}

Assume that the condition \texttt{[tiny-dom]} is not met. Necessarily,
  \begin{equation}\label{tangent case 1}
      \sum_{\vec{S}_{\ell} \in \vec{\Sc}_{\ell}}\Big\|\sum_{T \in \T[\vec{S}_{\ell}]} \chi_T\Big\|_{\BL{k,A_{\ell}}^{p_{\ell}}(B[\vec{S}_{\ell}])}^{p_{\ell}} \leq \,\frac{1}{2}\, \sum_{\vec{S}_{\ell} \in \vec{\Sc}_{\ell,\textrm{\texttt{tang}}}} \Big\|\sum_{T \in \T[\vec{S}_{\ell}]} \chi_T\Big\|_{\BL{k,A_{\ell}}^{p_{\ell}}(B[\vec{S}_{\ell}])}^{p_{\ell}},
\end{equation}
where the right-hand summation is restricted to those $S_{\ell} \in \vec{\Sc}_{\ell}$ for which \texttt{[alg 1]} does not terminate owing to \texttt{[tiny]} and therefore terminates owing to \texttt{[tang]}. Consequently, for each $\vec{S}_{\ell} \in \vec{\Sc}_{\ell,\textrm{\texttt{tang}}}$ the inequalities
  \begin{equation}\label{recursive step property 1}
      \Big\|\sum_{T \in \T[\vec{S}_{\ell}]} \chi_T\Big\|_{\BL{k,A_{\ell}}^{p_{\ell}}(B[\vec{S}_{\ell}])}^{p_{\ell}} \,\lessapprox \sum_{S_{\ell-1} \in \Sc_{\ell-1}[\vec{S}_{\ell}]} \Big\|\sum_{T \in \T[\vec{S}_{\ell-1}]}\chi_T\Big\|_{\BL{k,2A_{\ell-1}}^{p_{\ell}}(B[\vec{S}_{\ell-1}])}^{p_{\ell}},
\end{equation}
and
\begin{align}
\label{recursive step property 2}
    \sum_{S_{\ell-1} \in \Sc_{\ell-1}[\vec{S}_{\ell}]}  \#\T[S_{\ell-1}] &\,\lessapprox\, D_{\ell-1}^{1 + \varepsilon_{\;\!\!\circ}} \#\T[S_{\ell}]; \\
\label{recursive step property 3}
\max_{S_{\ell-1} \in \Sc_{\ell-1}[\vec{S}_{\ell}]} \#\T[S_{\ell-1}] &\,\lessapprox\,  D_{\ell-1}^{-(\ell-1) + \varepsilon_{\;\!\!\circ}} \#\T[S_{\ell}]
\end{align}
hold for some choice of:
\begin{itemize}
    \item Scale $\delta_{\ell-1}$ satisfying $\delta_{\ell} > \delta_{\ell-1} \geq \delta^{1-\varepsilon_{\;\!\!\circ}}$; non-admissible number $D_{\ell-1}$ and large integer $A_{\ell-1}$ satisfying $A_{\ell-1} \sim A_{\ell}$;
    \item Family $\Sc_{\ell-1}[\vec{S}_{\ell}]$ of $(\ell-1)$-dimensional transverse complete intersections of degree $O(1)$;
    \item Assignment of a subfamily $\T[\vec{S}_{\ell-1}] = \T[\vec{S}_{\ell}][S_{\ell-1}]$ of $\delta$-tubes for every $S_{\ell-1} \in \Sc_{\ell-1}[\vec{S}_{\ell}]$ such that each $T \in \T[\vec{S}_{\ell-1}]$ is tangent to $S_{\ell-1}$ on $B[\vec{S}_{\ell-1}]$.
\end{itemize}
Each inequality~\eqref{recursive step property 1},~\eqref{recursive step property 2} and~\eqref{recursive step property 3} is obtained by combining the definition of the stopping condition \texttt{[tang]} with Properties I, II and III from \texttt{[alg 1]}, respectively. Indeed, we take $$r_0 := \delta_{\ell},\quad \delta_{\ell-1} := \max\{r_{\!J}^{1+\varepsilon_{\;\!\!\circ}},\delta^{1-\varepsilon_{\;\!\!\circ}}\},\quad \text{and}\quad D_{\ell-1} := d^{\#_{\stc}(J)},$$ using the notation from \texttt{[alg 1]}.

The $\delta_{\ell-1}$, $D_{\ell-1}$ and $A_{\ell-1}$ can depend on the choice of $\vec{S}_{\ell}$, but this dependence can be essentially removed by pigeonholing. In particular,  $\#_{\stc}(J)$ depends on $\vec{S}_{\ell}$, but satisfies ${\#_{\stc}(J)}= O(\log \delta^{-1})$. Thus, since there are only logarithmically many possible different values, one may find a subset of the $\Sc_{\ell,\textrm{\texttt{tang}}}$ over which the $D_{\ell-1}$ all have a common value and, moreover, the inequality~\eqref{tiny case 1} still holds except that the constant $1/2$ is now replaced with, say,~$\delta^{-\varepsilon_{\;\!\!\circ}}$. A brief inspection of \texttt{[alg 1]} shows that both $\delta_{\ell-1}$ and $A_{\ell-1}$ are determined by~$D_{\ell-1}$ and so the desired uniformity is immediately inherited by these parameters. 

Letting $\vec{\Sc}_{\ell-1}$ denote the structured set 
\begin{equation*}
    \vec{\Sc}_{\ell-1} := \big\{ (\vec{S}_{\ell}, S_{\ell-1}) : \vec{S}_{\ell} \in \vec{\Sc}_{\ell,\textrm{\texttt{tang}}} \textrm{ and } S_{\ell-1} \in \Sc_{\ell-1}[\vec{S}_{\ell}] \big\},
\end{equation*}
where $\vec{\Sc}_{\ell,\textrm{\texttt{tang}}}$ is understood to be the refined collection described in the previous paragraph, it remains to verify that the desired properties hold for the newly constructed data. Property 2 follows immediately from~\eqref{recursive step property 2} and Property 3 from~\eqref{recursive step property 3}, so it remains only to verify Property 1.

By combining the inequality~\eqref{l step} from the previous stage of the algorithm with~\eqref{tangent case 1} and~\eqref{recursive step property 1}, one deduces that
   \begin{equation*}
    \Big\|\sum_{T \in \T} \chi_T\Big\|_{\BL{k,A}^p(\R^n)}  \,\lessapprox\,  C(\vec{D}_{\ell};\vec{\delta}_{\ell}) [\delta^n\#\T]^{1-\Theta_{\ell}} \Big\|\sum_{T \in \T[\vec{S}_{\ell-1}]} \chi_T\Big\|_{\ell^{p_{\ell}}\BL{k,2A_{\ell-1}}^{p_{\ell}}(\vec{\Sc}_{\ell-1})}^{\Theta_{\ell}}
\end{equation*}
where, for any $1 \leq q < \infty$ and $M \in \N$, we write
\begin{equation*}
    \Big\|\sum_{T \in \T[\vec{S}_{\ell-1}]} \chi_T\Big\|_{\ell^{q}\BL{k,M}^{q}(\vec{\Sc}_{\ell-1})} := \Bigg(\sum_{\vec{S}_{\ell-1}\in \vec{\Sc}_{\ell-1}}\!\Big\|\sum_{T \in \T[\vec{S}_{\ell-1}]} \chi_T\Big\|_{\BL{k,M}^{q}(B[\vec{S}_{\ell - 1}])}^{q}\Bigg)^{1/q}.
\end{equation*}
Taking $q = p_{\ell}$ and $M = 2A_{\ell-1}$, the logarithmic convexity inequality (Lemma~\ref{logarithmic convexity inequality lemma}) dominates the preceding expression by
\begin{equation*}
  \Big\|\sum_{T \in \T[\vec{S}_{\ell-1}]} \chi_T\Big\|_{\ell^{1}\BL{k,A_{\ell-1}}^{1}(\vec{\Sc}_{\ell-1})}^{1-\Theta_{\ell-1}/\Theta_\ell}\,  \Big\|\sum_{T \in \T[\vec{S}_{\ell-1}]} \chi_T\Big\|_{\ell^{p_{\ell-1}}\BL{k,A_{\ell-1}}^{p_{\ell-1}}(\vec{\Sc}_{\ell-1})}^{\Theta_{\ell-1}/\Theta_{\ell}}.
\end{equation*}
Observe that, trivially, one has
\begin{equation*}
 \Big\|\sum_{T \in \T[\vec{S}_{\ell-1}]} \chi_T\Big\|_{\ell^{1}\BL{k,A_{\ell-1}}^{1}(\vec{\Sc}_{\ell-1})}\, \lesssim\,\,\Big(\frac{\delta_{\ell-1}}{\delta}\Big) \delta^n\!\!\! \sum_{\vec{S}_{\ell-1} \in \vec{\Sc}_{\ell-1}} \# \T[\vec{S}_{\ell-1}].
\end{equation*}
and, by Property 2 for the tube families $\{\T[\vec{S}_{i}] : \vec{S}_i \in \vec{\Sc}_i\}$ for $\ell - 1 \leq i \leq n-1$, it follows that
\begin{equation*}
 \Big\|\sum_{T \in \T[\vec{S}_{\ell-1}]} \chi_T\Big\|_{\ell^{1}\BL{k,A_{\ell-1}}^{1}(\vec{\Sc}_{\ell-1})} \,\lesssim\, \Big(\frac{\delta_{\ell-1}}{\delta}\Big)  \Big( \prod_{i=\ell-1}^{n-1}D_{i}^{1 + \varepsilon_{\;\!\!\circ}}\Big) \delta^n \#\T.
\end{equation*}
One may readily verify that
\begin{equation*}
     C(\vec{D}_{\ell}; \vec{\delta}_{\ell})\cdot  \Big( \frac{\delta_{\ell-1}}{\delta} \prod_{i=\ell-1}^{n-1}D_{i}^{1 + \varepsilon_{\;\!\!\circ}}\Big)^{\Theta_\ell-\Theta_{\ell-1}} = C(\vec{D}_{\ell-1};\vec{\delta}_{\ell-1})
\end{equation*}
and so, combining the above estimates, 
   \begin{equation*}
     \Big\|\sum_{T \in \T} \chi_T\Big\|_{\BL{k,A}^p(\R^n)} \,\lessapprox\, C(\vec{D}_{\ell-1};\vec{\delta}_{\ell-1}) [
      \delta^n\#\T]^{1-\Theta_{\ell-1}} \Big\|\!\!\sum_{T \in \T[\vec{S}_{\ell-1}]} \!\! \chi_T  \Big\|_{\ell^{p_{\ell-1}}\BL{k,A_{\ell-1}}^{p_{\ell-1}}(\vec{\Sc}_{\ell-1})}^{\Theta_{\ell-1}},
\end{equation*}
which is Property 1.
\\

\paragraph{\underline{\texttt{The final stage}}} If the algorithm has not stopped by the $k$th step, then it necessarily terminates at the $k$th step. Indeed, otherwise~\eqref{l step} would hold for $\ell = k-1$ and families $\T[\vec{S}_{{k-1}}]$ of $\delta_{k-1}$-tubes which are tangent to some transverse complete intersection of dimension $k-1$. By the vanishing property of the $k$-broad norms as described in Lemma~\ref{vanishing lemma}, one would then have
\begin{equation*}
    \Big\|\sum_{T \in \T[\vec{S}_{k-1}]} \chi_T\Big\|_{\BL{k,A_{k-1}}^{p_{k-1}}(B[\vec{S}_{k-1}] )} = 0,
\end{equation*}
which, by \eqref{l step}, would contradict the non-degeneracy hypothesis~\eqref{non degeneracy hypothesis}.

Suppose the recursive process terminates at step $m$, so that $m \geq k$. For each $\vec{S}_m \in \vec{\Sc}_{m,\textrm{\texttt{tiny}}}$ let $\O[\vec{S}_m]$ denote the final collection of cells output by \texttt{[alg 1]} (that is, the collection denoted by $\O_{\!J}$ in the notation of the previous subsection) when applied to estimate the broad norm $\|\sum_{T \in \T[\vec{S}_m]}\chi_T\|_{\BL{k,A_m}^{p_m}(B[\vec{S}_m])}$. 
By Properties I, II and III of \texttt{[alg 1]} one has 
   \begin{equation*}
      \Big\|\sum_{T \in \T[\vec{S}_m]}\chi_T\Big\|_{\BL{k,A_m}^{p_m}(B[\vec{S}_m])}^{p_m} \,\lesssim \sum_{O \in \O[\vec{S}_m]}  \Big\|\sum_{T \in \T[O]} \chi_T \Big\|_{\BL{k,A_{m-1}}^{p_m}(O)}^{p_m},
\end{equation*}
for some $A_{m-1} \sim A_m$ where the families $\T[O]$ satisfy
\begin{equation}\label{tiny property 2}
    \sum_{O \in \O[\vec{S}_m]} \#\T[O] \,\lesssim\, D_{m-1}^{1+ \varepsilon_{\circ}}  \#\T[\vec{S}_m]
    \end{equation}
and
\begin{equation}\label{tiny property 3}
    \max_{O \in \O[\vec{S}_m]} \#\T[O] \,\lesssim\, D_{m-1}^{-(m-1) + \varepsilon_{\circ}} \#\T[\vec{S}_m]
\end{equation}
for $D_{m-1}$ a large and (in general) non-admissible parameter. Once again, by pigeonholing, one may pass to a subcollection of $\Sc_{m, \textrm{\texttt{tiny}}}$ and thereby assume that the $D_{m-1}$ (and also the $A_{m-1}$) all share a common value. 

If $\O$ denotes the union of the $\O[\vec{S}_m]$ over all $\vec{S}_m$ belonging to subcollection of $\Sc_{m, \textrm{\texttt{tiny}}}$ described above, then \texttt{[alg 2]} outputs the following inequality.
\begin{first key estimate} 
   \begin{equation*}
     \Big\|\sum_{T \in \T} \chi_T\Big\|_{\BL{k,A}^p(\R^n)} \lessapprox\,  C(\vec{D}_m; \vec{\delta}_m) [\delta^n \#\T]^{1-\Theta_m}  \Bigg(\sum_{O\in\O}\Big\|\sum_{T \in \T[O]} \chi_T\Big\|^{p_m}_{\BL{k,A_{m-1}}(O)}\Bigg)^{\frac{\Theta_m}{p_m}}\!\!.
\end{equation*}
\end{first key estimate}

\section{Proof of Theorem~\ref{main theorem}}\label{proof section}

Henceforth, fix $\T$ to be a family of $\delta$-tubes in $\R^n$ which satisfy the $(D, N)$-polynomial Wolff axiom for some $D$, chosen sufficiently large (depending only on the admissible parameters $n$ and $\varepsilon$) so as to satisfy the forthcoming requirements of the proof. Without loss of generality, we may assume that~$\T$ satisfies the non-degeneracy hypothesis~\eqref{non degeneracy hypothesis}. The algorithms described in the previous section can be applied to this tube family, leading to the final decomposition of the broad norm described in the first key estimate. One therefore wishes to show, using the polynomial Wolff axiom hypothesis, that the quantity on the right-hand side of the first key estimate can be effectively bounded, provided that the exponents $p_k, \dots, p_n$ are suitably chosen. 

Since each $O \in \O$ is contained in a ball of radius at most $\delta^{1-\varepsilon_{\;\!\!\circ}}$, trivially one may bound
\begin{equation*}
    \Big\|\sum_{T \in \T[O]} \chi_T\Big\|_{\BL{k,A_{m-1}}^{p_m}(O)}^{p_m} \,\lessapprox\, \delta^{n} \big(\#\T[O]\big)^{p_m}.
\end{equation*}
Recalling that 
$
  \Theta_{m}(1-\frac{1}{p_{m}}) = 1-\frac{1}{p},
$ this yields
\begin{equation*}
\Bigg(\sum_{O\in\O}\Big\|\sum_{T \in \T[O]} \chi_T\Big\|^{p_m}_{\BL{k,A_{m-1}}(O)}\Bigg)^{\frac{\Theta_m}{p_m}}\lessapprox\,   \big(\max_{O \in \O}\#\T[O]\big)^{1-\frac{1}{p}}\Big( \delta^{n}\sum_{O \in \O}  \#\T[O]\Big)^{\frac{\Theta_m}{p_m}}.
\end{equation*}
Now \eqref{tiny property 2} and repeated application of Property 2 from \texttt{[alg 2]} imply 
\begin{equation*}
 \sum_{O \in \O} \#\T[O] \,\lessapprox\, \Big(\prod_{i = m-1}^{n-1} D_{i}^{1 + \varepsilon_{\;\!\!\circ}}\Big)\#\T.  
\end{equation*}
 Combining this with the first key estimate and the definition of $C(\vec{D}_m;\vec{\delta}_m)$, one concludes that 
  \begin{equation}\label{reduction to max bound}
    \Big\|\sum_{T \in \T} \chi_T\Big\|_{\BL{k,A}^{p}(\R^n)}  \,\lessapprox\,  \mathbf{C}(\vec{D}; \vec{\delta}\,)  \big(\max_{O \in \O}\#\T[O]\big)^{1 - \frac{1}{p}} \Big(  \delta\sum_{T \in \T} |T| \Big)^{\frac{1}{p}}
\end{equation}
where, taking $\delta_{m-1} := \delta$, the constant takes the form
\begin{equation*}
    \mathbf{C}(\vec{D}; \vec{\delta}\,) := \prod_{i = m-1}^{n-1} \Big(\frac{\delta_{i}}{\delta}\Big)^{\Theta_{i+1}-\Theta_{i}}D_{i}^{\Theta_{i+1} - (1 - \frac{1}{p}) + O(\varepsilon_{\;\!\!\circ})}.
\end{equation*}

In order to bound the maximum appearing on the right-hand side of~\eqref{reduction to max bound}, by~\eqref{tiny property 3} and repeated application of Property 3 of \texttt{[alg 2]}, it follows that 
\begin{equation*}
\max_{O \in \O} \#\T[O] \,\lessapprox\, \Big( \prod_{i=m-1}^{\ell-1}D_{i}^{-i + \varepsilon_{\;\!\!\circ}}\Big) \max_{\vec{S}_{\ell} \in \vec{\Sc}_{\ell}}\#\T[\vec{S}_{\ell}]
\end{equation*}
whenever $m \leq \ell \leq n$.
Recall, for each tube family $\T[\vec{S}_{\ell}]$ produced by \texttt{[alg 2]} there exists a $\delta_{\ell}$-ball $B_{\delta_{\ell}} := B[\vec{S}_{\ell}]$ such that every $\delta$-tube $T \in \T[\vec{S}_{\ell}]$ is tangent to~$S_{\ell}$ in~$B_{\;\!\!\delta_{\ell}}$; in particular, 
\begin{equation*}
    T \cap B_{\;\!\!\delta_{\ell}} \cap N_{\;\!\!\delta}S_{\ell} \neq \emptyset \quad \textrm{and} \quad T \cap 2 B_{\;\!\!\delta_{\ell}} \subseteq N_{\;\!\!2\delta} S_{\ell}.
\end{equation*}
Here $S_{\ell}$ is a transverse complete intersection of dimension $\ell$ and $\overline{\deg}S_{\ell}$ depends only on the admissible parameters $n$ and $\varepsilon$. Thus, if $D$ is chosen to be sufficiently large, then the $(D,N)$-polynomial Wolff axiom implies that
\begin{equation*}
    \#\T[\vec{S}_{\ell}] \leq \#\big\{ T \in \T : |T \cap 2B_{\;\!\!\delta_{\ell}} \cap N_{\;\!\!2\delta} S_{\ell}| \gtrsim \delta_{\ell} |T|  \big\} \lesssim N\delta^{-(n-1)}\delta_{\ell}^{-n}|2B_{\;\!\!\delta_{\ell}}\cap N_{\;\!\!2\delta} S_{\ell}| .
\end{equation*}
Moreover, by Wongkew's lemma~\cite{Wongkew1993},
\begin{equation*}
   |2 B_{\;\!\!\delta_{\ell}} \cap N_{\;\!\!2\delta} S_{\ell}| \,\lesssim\,  \delta^{n-\ell} \delta_{\ell}^{\ell}, 
\end{equation*}
Combining these observations,
\begin{equation*}
   \max_{\vec{S}_{\ell} \in \vec{\Sc}_{\ell} } \# \T[\vec{S}_{\ell}] \,\lesssim\, N \delta^{-(n-1)}\Big( \frac{\delta_\ell}{\delta}\Big)^{-(n-\ell)}
\end{equation*}
so that
\begin{equation*}
    \max_{O \in \O}\#\T[O] \,\lessapprox\, N \Big( \prod_{i = m-1}^{\ell-1}D_{i}^{-i + \varepsilon_{\;\!\!\circ}}\Big)\Big( \frac{\delta_\ell}{\delta}\Big)^{-(n-\ell)} \delta^{-(n-1)}
\end{equation*}
for all $m \leq \ell \leq n$. Finally, these $n-m+1$ different estimates can be combined into a single inequality by taking a weighted geometric mean, yielding:

\begin{second key estimate} Let $0 \leq \gamma_m, \dots, \gamma_n \leq 1$ satisfy $\sum_{j = m}^{n} \gamma_{j} = 1$. Then
\begin{equation*}
    \max_{O \in \O} \#\T[O] \,\lessapprox\, N \Big(\prod_{i=m-1}^{n-1}\Big(\frac{\delta_i}{\delta}\Big)^{-(n-i)\gamma_{i}} D_{i}^{-i(1-\sum_{j=m}^{i} \gamma_j) + O(\varepsilon_{\;\!\!\circ})} \Big) \delta^{-(n-1)}.
\end{equation*}
\end{second key estimate}

Substituting the second key estimate into~\eqref{reduction to max bound}, one obtains
  \begin{equation*}
      \Big\|\sum_{T \in \T} \chi_T\Big\|_{\BL{k,A}^{p}(\R^n)} \,\lessapprox\, N^{1 - \frac{1}{p}} \Big(\prod_{i = m-1}^{n-1} \Big(\frac{\delta_{i}}{\delta}\Big)^{X_{i}}D_{i}^{Y_{i} + O(\varepsilon_{\;\!\!\circ})}\Big) \delta^{-(n-1-\frac{n}{p})}
      \Big(\sum_{T \in \T} |T| \Big)^{\frac{1}{p}} 
\end{equation*}
where
\begin{align*}
    X_{i}&:= \Theta_{i+1}-\Theta_{i}-(n-i)\gamma_{i}\Big(1-\frac{1}{p}\Big); \\
    Y_{i}&:= \Theta_{i+1} - \Big(1+i\big(1-\sum_{j=m}^{i} \gamma_j\big)\Big)\Big(1-\frac{1}{p}\Big).
\end{align*}
One now chooses the various exponents so that $X_{i}, Y_{i} = 0$ for all $m \leq i \leq n-1$ and $Y_{m-1} = 0$. This ensures that the $(\delta_i/\delta)^{X_i}$ and $D_i^{Y_i}$ factors in the above expression are admissible but does not allow one to control the $D_i^{O(\varepsilon_{\;\!\!\circ})}$ factors, which may still be non-admissible. To deal with the $D_i^{O(\varepsilon_{\;\!\!\circ})}$ one may perturb the~$p$ exponent which results under the conditions $X_{i}, Y_{i} = 0$, so that $Y_{i}$ becomes negative, and then choose $\varepsilon_{\;\!\!\circ}$ sufficiently small depending on the choice of perturbation. This yields an open range of $k$-broad estimates, which can then be trivially extended to a closed range via interpolation through logarithmic convexity (the interpolation argument relies on the fact that one is permitted an $\delta^{-\varepsilon}$-loss in the constants in the $k$-broad inequalities).

The condition $X_{i} = 0$ is equivalent to
\begin{equation}\label{our r constraint}
\Big(1 - \frac{1}{p_{i+1}}\Big)^{-1}-\Big(1 - \frac{1}{p_{i}}\Big)^{-1} = (n-i)\gamma_{i}
\end{equation}
whilst the condition $Y_{i-1} = 0$ is equivalent to 
\begin{equation}\label{our D constraint}
    \Big(1 - \frac{1}{p_{i}}\Big)^{-1} = i - (i-1)\sum_{j=m}^{i-1} \gamma_j.
\end{equation}
Choose $p_m := \frac{m}{m -1}$ so that~\eqref{our D constraint} holds in the $i = m$ case. The remaining~$p_{i}$ are then defined in terms of the $\gamma_j$ by the equation
\begin{equation}\label{linear system 1} \Big(1 - \frac{1}{p_{i}}\Big)^{-1}=m + \sum_{j=m}^{i-1}(n-j)\gamma_j
\end{equation} 
so that each of the $n-m$ constraints in~\eqref{our r constraint} is met. 

It remains to solve for the $n-m+1$ variables $\gamma_m, \dots, \gamma_n$. By comparing the right-hand sides of~\eqref{our D constraint} and~\eqref{linear system 1}, it follows that
\begin{equation}\label{linear system 2}
  \sum_{j = m}^{i-1} (n+i - j - 1)\gamma_j = i - m  \qquad \textrm{for $m+1 \leq i \leq n$.}
\end{equation}
To solve this linear system, let $\beta_{i}$ denote the left-hand side of \eqref{linear system 2} and observe that 
\begin{equation*}
    \beta_{i+1} + \beta_{i-1} - 2\beta_{i} =  n\gamma_{i}-(n - 1)\gamma_{i-1} \qquad\textrm{ for $m+1 \leq i \leq n-1$,}
\end{equation*}
where $\beta_m := 0$. On the other hand, by considering the right-hand side of~\eqref{linear system 2}, it is clear that $\beta_{i+1} + \beta_{i-1} - 2\beta_{i} = 0$. Combining these observations gives a recursive relation for  $\gamma_j$ and from this one deduces that 
\begin{equation*}
    \gamma_j = \frac{1}{n} \prod_{i=m}^{j-1} \frac{n-1}{n} = \frac{1}{n}\Big(1-\frac{1}{n}\Big)^{{j-m}} \qquad \textrm{for $m \leq j \leq n-1$.}
\end{equation*}

It remains to check that these parameter values give the correct value of $p_n$, corresponding to the exponent featured in Theorem~\ref{main theorem}. It follows from~\eqref{our D constraint} that
\begin{align*}
    \Big(1 - \frac{1}{p_n}\Big)^{-1} &= n -(n-1)\sum_{j=m}^{n-1} \frac{1}{n}\Big(1-\frac{1}{n}\Big)^{{j-m}} \\
     &= 1+(n-1)\Big(1-\frac{1}{n}\Big)^{n-m}.
\end{align*}
This is largest when $m=k$, which directly yields the desired range of $p$, as stated in Theorem~\ref{main theorem}, completing the proof. 





\begin{thebibliography}{10}

\bibitem{Bennett2014} J. Bennett, Aspects of multilinear harmonic analysis related to transversality, in {\it Harmonic Analysis and PDE}, 1--28, Contemp. Math. \textbf{612}, Amer. Math. Soc., Providence, RI. 

\bibitem{BCT2006}
J. Bennett, A. Carbery, and T. Tao, \emph{On the multilinear
  restriction and {K}akeya conjectures}, Acta Math. \textbf{196} (2006), no.~2,
  261--302. 


\bibitem{Bourgain1991a}  J. Bourgain, \emph{Besicovitch-type maximal operators and applications to Fourier analysis}, Geom. Funct. Anal. \textbf{22} (1991), no.~2, 147--187.

  


\bibitem{Bourgain1999}
\bysame, \emph{On the dimension of {K}akeya sets and related maximal
  inequalities}, Geom. Funct. Anal. \textbf{9} (1999), no.~2, 256--282.

\bibitem{BG2011}
J. Bourgain and L. Guth, \emph{Bounds on oscillatory integral operators
  based on multilinear estimates}, Geom. Funct. Anal. \textbf{21} (2011),
  no.~6, 1239--1295. 

\bibitem{CDR1986} M. Christ,  J. Duoandikoetxea, and J. L. Rubio  de  Francia,
\emph{Maximal  operators  associated
to the Radon transform and the Calder\'on-Zygmund method of rotations}, Duke Math. J. \textbf{53}
(1986), no.~1, 189--209.

\bibitem{Cordoba1977}
A. C\'{o}rdoba, \emph{The {K}akeya maximal function and the spherical
  summation multipliers}, Amer. J. Math. \textbf{99} (1977), no.~1, 1--22.
  


\bibitem{Davies1971}
R. O. Davies, \emph{Some remarks on the {K}akeya problem}, Proc. Cambridge
  Philos. Soc. \textbf{69} (1971), 417--421. 

\bibitem{Dvir2009}
Z. Dvir, \emph{On the size of {K}akeya sets in finite fields}, J. Amer. Math.
  Soc. \textbf{22} (2009), no.~4, 1093--1097. 

\bibitem{GR}
B. Green and I. Z. Ruzsa, \emph{{O}n the arithmetic {K}akeya conjecture of
  {K}atz and {T}ao}, Preprint: arXiv:1712.02108.



\bibitem{Guth2010}
L. Guth, \emph{The endpoint case of the {B}ennett-{C}arbery-{T}ao multilinear
  {K}akeya conjecture}, Acta Math. \textbf{205} (2010), no.~2, 263--286.

\bibitem{Guth2016b}
\bysame, \emph{Degree reduction and graininess for {K}akeya-type sets in
  {$\mathbb{R}^3$}}, Rev. Mat. Iberoam. \textbf{32} (2016), no.~2, 447--494.

\bibitem{Guth2016}
\bysame, \emph{A restriction estimate using polynomial partitioning}, J. Amer.
  Math. Soc. \textbf{29} (2016), no.~2, 371--413. 
  
  \bibitem{Guth2019}
\bysame, \emph{Restriction estimates using polynomial partitioning {II}},
  Acta Math. \textbf{221} (2018), 81--142.

\bibitem{GK2015}
L. Guth and N. H. Katz, \emph{On the {E}rd\"os distinct distances
  problem in the plane}, Ann. of Math. (2) \textbf{181} (2015), no.~1,
  155--190. 
  
  \bibitem{GHI}
L. Guth, J. Hickman, and M. Iliopoulou, \emph{Sharp estimates for oscillatory
  integral operators via polynomial partitioning}, Preprint: arXiv:1710.10349.

\bibitem{GZ2018}
L. Guth and J. Zahl, \emph{Polynomial {W}olff axioms and {K}akeya-type
  estimates in {$\mathbb{R}^4$}}, Proc. Lond. Math. Soc. (3) \textbf{117}
  (2018), no.~1, 192--220. 

\bibitem{HR2018}
J. Hickman and K. M. Rogers, \emph{Improved {F}ourier restriction
  estimates in higher dimensions}, Preprint: arXiv:1807.10940.

\bibitem{KLT2000}
N. H. Katz, I. \L aba, and T. Tao, \emph{An improved bound on
  the {M}inkowski dimension of {B}esicovitch sets in {${\mathbb{R}}^3$}}, Ann. of
  Math. (2) \textbf{152} (2000), no.~2, 383--446. 

\bibitem{KR2018}
N. H. Katz and K. M. Rogers, \emph{On the polynomial Wolff axioms},
  Geom. Funct. Anal. \textbf{28} (2018), 1706--1716.

\bibitem{KT1999}
N. H. Katz and T. Tao, \emph{Bounds on arithmetic projections, and
  applications to the {K}akeya conjecture}, Math. Res. Lett. \textbf{6} (1999),
  no.~5-6, 625--630. 


\bibitem{KT2002}
\bysame, \emph{New bounds for {K}akeya problems}, J. Anal. Math. \textbf{87}
  (2002), 231--263, Dedicated to the memory of Thomas H. Wolff. 

\bibitem{KT2002b}
\bysame, \emph{Recent progress on the Kakeya conjecture}, in Harmonic Analysis and Partial Differential Equations (El Escorial, 2000). Publ. Mat. 2002, 161--179.
\bibitem{KZ2019}
N. H. Katz and J. Zahl, \emph{An improved bound on the {H}ausdorff
  dimension of {B}esicovitch sets in {$\mathbb{R}^3$}}, J. Amer. Math. Soc.
  \textbf{32} (2019), no.~1, 195--259. 

\bibitem{Maple}
{Maplesoft, a division of Waterloo Maple Inc.}, \emph{Maple 17}.

\bibitem{Matousek} J. Matou\v{s}ek, \emph{Using the {B}orsuk-{U}lam theorem},
   {Universitext},
Springer-Verlag, Berlin,
     2003,
    xii+196.

\bibitem{Schlag1998} W. Schlag, \emph{A geometric inequality with applications to the {K}akeya
              problem in three dimensions}, Geom. Funct. Anal.
 \textbf{8} (1998), no.~3, 606--625.
   

\bibitem{TS2012}
J. Solymosi and T. Tao, \emph{An incidence theorem in higher
  dimensions}, Discrete Comput. Geom. \textbf{48} (2012), no.~2, 255--280.
  
\bibitem{Stone1942} A. H. Stone and J. W. Tukey, \emph{Generalized ``sandwich'' theorems}, Duke Math. J.
 \textbf{9} (1942), 356--359.

\bibitem{Tao2005}
T. Tao, \emph{A new bound for finite field {B}esicovitch sets in four
  dimensions}, Pacific J. Math. \textbf{222} (2005), no.~2, 337--363.

\bibitem{TVV1998}
T. Tao, A. Vargas, and L. Vega, \emph{A bilinear approach to the
  restriction and {K}akeya conjectures}, J. Amer. Math. Soc. \textbf{11}
  (1998), no.~4, 967--1000. 

\bibitem{Wang2018}
H. Wang, \emph{A restriction estimate in {$\mathbb{R}^3$} using brooms},
  Preprint: arXiv:1802.04312.

\bibitem{Wolff1995}
T. Wolff, \emph{An improved bound for {K}akeya type maximal functions},
  Rev. Mat. Iberoamericana \textbf{11} (1995), no.~3, 651--674. 

\bibitem{Wolff1999}
\bysame, \emph{Recent work connected with the {K}akeya problem}, Prospects in
  mathematics ({P}rinceton, {NJ}, 1996), Amer. Math. Soc., Providence, RI,
  1999, pp.~129--162. 

\bibitem{Wongkew1993} R. Wongkew,
    \emph{Volumes of tubular neighbourhoods of real algebraic varieties},
  Pacific J. Math. \textbf{159} (1993), no.~1, 177--184.

\bibitem{Zahl2018}
J. Zahl, \emph{A discretized {S}everi-type theorem with applications to
  harmonic analysis}, Geom. Funct. Anal. \textbf{28} (2018), no.~4, 1131--1181.

\end{thebibliography}
\end{document}